\documentclass[12pt]{amsart}

\usepackage{amssymb, amsthm, amsmath, amsrefs}

\newtheorem{theorem}{Theorem}[section]

\newtheorem{lemma}[theorem]{Lemma}
\newtheorem{prop}[theorem]{Proposition}
\newtheorem{corollary}[theorem]{Corollary}

\newtheorem{question}[theorem]{Question}

\theoremstyle{definition}

\newtheorem{notation}[theorem]{Notation}
\newtheorem{example}[theorem]{Example}

\newtheorem{claim}{Claim}

\theoremstyle{remark}
\newtheorem{remark}[theorem]{Remark}

\numberwithin{equation}{section}

\newcommand{\D}{\mathbb{D}}
\newcommand{\cD}{\overline{\D}}
\newcommand{\E}{\mathbb{E}}
\newcommand{\C}{\mathbb{C}}
\newcommand{\T}{\mathbb{T}}
\newcommand{\Z}{\mathbb{Z}}

\newcommand{\ip}[2]{\langle #1, #2 \rangle}
\newcommand{\mcH}{\mathcal{H}}
\newcommand{\mcK}{\mathcal{K}}

\newcommand{\mcQ}{\mathcal{Q}}
\newcommand{\mcP}{\mathcal{P}}
\newcommand{\mcR}{\mathcal{R}}
\newcommand{\mcV}{\mathcal{V}}

\newcommand{\Kphi}{\mcK_{\phi}}

\newcommand{\Hphi}{\mcH_{\phi}}
\newcommand{\supp}{\text{supp}}

\newcommand\xqed[1]{%
  \leavevmode\unskip\penalty9999 \hbox{}\nobreak\hfill
  \quad\hbox{#1}}
\newcommand\EOEx{\xqed{$\diamond$}}

\newcommand{\kgeq}{\succcurlyeq}
\newcommand{\kleq}{\preccurlyeq}

\newcommand{\Pzp}{\mcP_{d,p}^0}
\newcommand{\Pop}{\mcP_{d,p}^1}
\newcommand{\Ptp}{\mcP_{d,p}^2}

\title[Inner functions on the bidisk]{Inner functions on the bidisk
  and associated Hilbert spaces} 
\author{Kelly Bickel}
\address{Washington University in St. Louis, St. Louis, MO, 63130}
\email{kbickel@math.wustl.edu}

\author{Greg Knese}
\address{University of Alabama, Tuscaloosa, AL, 35487-0350}
\email{geknese@bama.ua.edu}

\date{\today}

\keywords{inner function, bidisk, polydisk, bidisc, polydisc, Hardy
  space, Agler decomposition}

\thanks{GK supported by NSF grant DMS-1048775}

\subjclass{Primary 47A57; Secondary 30C15, 42B05}

\begin{document}
\bibliographystyle{apalike}

\begin{abstract} 
Matrix valued inner functions on the bidisk have a number of natural
subspaces of the Hardy space on the torus associated to them.  We
study their relationship to Agler decompositions, regularity up to the
boundary, and restriction maps into one variable spaces.  We give a
complete description of the important spaces associated to matrix
rational inner functions. The dimension of some of these spaces can be
computed in a straightforward way, and this ends up having an
application to the study of three variable rational inner functions.
Examples are included to highlight the differences between the scalar
and matrix cases.
\end{abstract} 

\maketitle

\section{Introduction}

Inner functions $\phi$ on the unit disk $\D$, their associated Hilbert
spaces $\phi H^2$ and $\mcH_{\phi} = H^2 \ominus \phi H^2$, and the
associated shift $S$ and backward shift $S^*$ operators on these
spaces form a natural and rich area of analysis.  Natural, because by
Beurling's theorem \cite{aB48}, every invariant subspace of the
forward shift on $\ell^2(\mathbb{N})$ is unitarily equivalent to $\phi
H^2$ for some inner $\phi$.  Rich, because if we allow $\phi$ to be
operator valued, then any contractive operator on a separable Hilbert
space can be modeled as $S^*$ on $\mcH_{\phi}$ for some $\phi$.  At
the same time, simple choices for $\phi$ provide interesting examples.
If $\phi$ is a finite Blaschke product, the space $\mcH_{\phi}$ is
finite dimensional and is related to orthogonal polynomials on the
unit circle.  If $\phi(z) = \exp \left( a \frac{z+1}{z-1}\right)$,
$a>0$, then $\mcH_{\phi}$ is isometric to the Paley-Wiener space $PW_a
= \mathcal{F}(L^2(0,a))$ through a change of coordinates to the upper
half-plane.  See \cite{hH64} or \cite{NN86} for the one variable
theory.

Inner functions on the bidisk $\D^2 = \D \times \D$ and their
associated Hilbert spaces are far richer and considerably less well
developed than their one variable counterparts.  For early work on the
topic see for instance Rudin \cite{wR69}, Ahern-Clark \cite{AC69},
Ahern \cite{pA72}, and Sawyer \cite{eS79}.  Rational inner functions
on the bidisk have close ties to the study of stable bivariate
polynomials (e.g. polynomials with no zeros on the bidisk), and
Hilbert space methods have proved useful in understanding them.  See
Cole-Wermer \cite{CW99}, Geronimo-Woerdeman \cite{GW04},
Ball-Sadosky-Vinnikov \cite{BSV05}, Woerdeman \cite{hW10}, Knese
\cite{gK10}, and Geronimo-Iliev-Knese \cite{GIK12}.  Any type of
general classification of inner functions on the bidisk or polydisk
seems unknown and difficult.

Recall that $\phi:\D^2 \to \overline{\D}$ is an inner function if $\phi$ is
holomorphic and satisfies
\[
\lim_{r\nearrow 1} |\phi(re^{i\theta_1},re^{i\theta_2})|=
|\phi(e^{i\theta_1}, e^{i\theta_2})| = 1 \text{ a.e. }
\]
We also use the term inner function for holomorphic functions $\phi:
\D^2 \to \mathcal{B}_1$, where $\mathcal{B}_1$ is the closed unit ball
in the operator norm of the bounded linear operators from a separable
Hilbert space $\mcV$ to itself such that
\[
\phi(z)^* \phi(z) = \phi(z)\phi(z)^*=I \text{ for a.e. } z \in \T^2 =
(\partial \D)^2,
\]
i.e. $\phi$ is unitary valued almost everywhere on the torus. Note
that the radial boundary limits of these operator valued functions
converge in the strong operator topology:
\[
\lim_{r\nearrow 1} \phi(re^{i\theta_1},re^{i\theta_2}) v = \phi(e^{i\theta_1},e^{i\theta_2})v
\]
for each $v \in \mcV$ and for a.e. $(\theta_1,\theta_2) \in
[0,2\pi]^2$.

Let $Z_1, Z_2$ denote the coordinate functions $Z_j(z_1,z_2) = z_j$.
Let us define some standard subspaces of $L^2 =L^2(\T^2) \otimes \mcV$
according to their Fourier series support. Let $\Z_+ = \{0,1,2\dots\},
\Z_{-} = \{-1,-2,-3,\dots\}$.  If $N\subset \Z^2$ and $f \in L^2$, the
statement $\supp(\hat{f}) \subset N$ means $\hat{f}(j_1,j_2) = 0$ for
$(j_1,j_2) \notin N$. We caution that mention of $\mcV$ is suppressed
throughout our definitions in order to keep the notation uncluttered.

\[
\begin{aligned}
H^2 &= \{f \in L^2: \supp(\hat{f}) \subset \Z_+^2\} \\ 
L^2_{+\bullet} &= \{f \in L^2: \supp(\hat{f}) \subset \Z_+ \times \Z\} \\
L^2_{\bullet +} &= \{f \in L^2: \supp(\hat{f}) \subset \Z \times \Z_+\} \\
L^2_{-\bullet} &= \{f \in L^2: \supp(\hat{f}) \subset \Z_{-} \times \Z\} \\
L^2_{\bullet -} &= \{f \in L^2: \supp(\hat{f}) \subset \Z \times \Z_{-}\} \\
L^2_{+-} &= \{f \in L^2: \supp(\hat{f}) \subset \Z_{+} \times \Z_{-}\} \\
L^2_{-+} &= \{f \in L^2: \supp(\hat{f}) \subset \Z_{-} \times \Z_+\} \\
L^2_{--} &= \{f \in L^2: \supp(\hat{f}) \subset \Z_{-} \times \Z_{-}\}.
\end{aligned}
\]

The important vector valued Hilbert spaces associated to $\phi$ are
\[
\begin{aligned}
\Hphi &= H^2 \ominus \phi H^2  = H^2 \cap \phi (L^2_{-+}\oplus L^2_{+-} \oplus L^2_{--})\\
\Hphi^1 &= H^2\cap \phi L^2_{\bullet -}\\
  \Hphi^2 &= H^2\cap \phi L^2_{-\bullet} \\
\Kphi &= H^2\cap \phi L^2_{--} = \Hphi^1
\cap \Hphi^2 \\
\mcK^{1}_\phi &= H^2\cap Z_1\phi L^2_{--} \\
 \mcK^{2}_\phi &= H^2\cap Z_2 \phi L^2_{--}.
\end{aligned}
\]

\begin{example} The basic example $\phi(z) = z_1^2z_2$
  should help make these definitions more concrete.  In this case,
  letting $\vee$ denote closed linear span in $H^2$, we have
\[
\begin{aligned} 
\Hphi &= \vee\{Z_1^j Z_2^k: j,k\geq 0, \text{ and } j \leq 1 \text{ or
} k = 0\} \\ \Hphi^1 &= \vee\{ Z_1^j : j \geq 0\} \\ \Hphi^2 &= \vee\{
Z_1^jZ_2^k: k \geq 0, j=0,1\} \\ \Kphi &= \vee\{ 1,Z_1\} \\ \Kphi^1 &=
\vee\{ 1, Z_1, Z_1^2\} \\ \Kphi^2 &= \vee\{1,Z_1,Z_2,Z_1Z_2\}.
\end{aligned}
\]
\EOEx
\end{example}

The space $\Hphi$ \emph{seems} to be the most natural generalization
of the one variable space $H^2(\T) \ominus \phi H^2(\T)$.  For
instance, in one variable, the reproducing kernel for $H^2 \ominus
\phi H^2$ is given by
\[
\frac{1-\phi(z)\phi(w)^*}{1-z\bar{w}},
\]
while in two variables the reproducing kernel of $\Hphi$ is
\[
\frac{1 - \phi(z)
  \phi(w)^* }{(1-z_1\bar{w}_1)(1-z_2\bar{w}_2)}.
\] 
Unfortunately, this fact is not as illuminating as the one variable
formula. The space $\Hphi$ can be broken down into an orthogonal
direct sum of various spaces above in a non-obvious way. This leads to
a more useful formula for the reproducing kernel, a result of the work
in Ball-Sadosky-Vinnikov \cite{BSV05} (see also \cite{gK11}). 

\begin{notation} \label{efgnotation}
Let $\phi: \D^2 \to \mathcal{B}_1$ be inner.  For $j=1,2$, let
\[
\begin{aligned}
E^j &= \text{ the reproducing kernel  for } \Hphi^j \ominus Z_j \Hphi^j\\
F^j &= \text{ the reproducing kernel for } (\Hphi^j\ominus
\Kphi)\ominus Z_j(\Hphi^j\ominus \Kphi)\\
G &= \text{ the reproducing kernel for } \Kphi.
\end{aligned}
\]
\end{notation}

\begin{theorem} \label{thm:fundagler}
If $\phi: \D^2 \to \mathcal{B}_1$ is inner, then using Notation
\ref{efgnotation}

\begin{equation} \label{agdecomp}
\begin{aligned}
\frac{1-\phi(z)\phi(w)^*}{(1-z_1\bar{w}_1)(1-z_2\bar{w}_2)} &= 
\frac{E^1(z,w)}{1-z_1\bar{w}_1} + \frac{F^2(z,w)}{1-z_2\bar{w}_2}  \\
&= \frac{F^1(z,w)}{1-z_1\bar{w}_1} + \frac{F^2(z,w)}{1-z_2\bar{w}_2} +G(z,w).
\end{aligned}
\end{equation}
\end{theorem}

 We shall outline a proof along the lines of Bickel \cite{kB11}, since
 various parts of the proof will be useful later.  We will see later
 in Propositions \ref{prop:sub1} and \ref{prop:sub2} that the spaces
 in Notation \ref{efgnotation} can be rewritten more simply
\[
\begin{aligned}
&\Hphi^j \ominus Z_j \Hphi^j = \Kphi^j \ominus Z_j \Kphi \\
& (\Hphi^j\ominus \Kphi)\ominus
Z_j(\Hphi^j\ominus \Kphi) = \Kphi^j \ominus \Kphi.
\end{aligned}
\]
Because of this, understanding $\phi$ boils down to understanding
$\Kphi, \Kphi^1\ominus \Kphi,$ and $\Kphi^2\ominus \Kphi$.  \emph{This
  is one of the most important themes of the paper.}

Rewriting the formula above we get
\begin{equation} \label{eqagdecomp}
  1-\phi(z)\phi(w)^* = (1-z_1\bar{w}_1)F^2(z,w) + (1-z_2\bar{w}_2)E^1(z,w),
\end{equation}
which is called an \emph{Agler decomposition} of $\phi$.  By symmetry
we also have
\begin{equation} \label{eqagdecompb}
  1-\phi(z)\phi(w)^* = (1-z_1\bar{w}_1)E^2(z,w) + (1-z_2\bar{w}_2)F^1(z,w).
\end{equation}

 In the scalar case, one can deduce Agler's Pick interpolation theorem
 on the bidisk (see \cite{AM02}) as well as And\^{o}'s inequality from
 operator theory from this formula via an approximation argument (see
 \cite{gK10}).

Because of this connection, any pair of holomorphic positive
semidefinite kernels $(A^1,A^2)$ satisfying
\begin{equation} \label{aglerkernels}  
1-\phi(z)\phi(w)^* = (1-z_1\bar{w}_1)A^2(z,w) +
(1-z_2\bar{w}_2)A^1(z,w) 
\end{equation}
for all $z,w \in \D^2$ are called \emph{Agler kernels} of $\phi$.
Labelling $A^2$ as the kernel next to $(1-z_1\bar{w}_1)$ seems to be
more natural in light of \eqref{agdecomp} and \eqref{eqagdecomp}. The
reproducing kernel Hilbert spaces associated to $A^1,A^2$ are denoted
$\mcH(A^1)$ and $\mcH(A^2).$ Any such pair can be characterized in
terms of the canonical kernels given in Theorem
\ref{thm:fundagler}. This characterization generalizes a similar
result of Ball-Sadosky-Vinnikov in \cite{BSV05}.

To set things up, if we equate the right sides of \eqref{eqagdecomp},
\eqref{eqagdecompb}, and \eqref{aglerkernels} one can derive
 \[
 G^1(z,w):= \frac{A^1(z,w) - F^1(z,w)}{1-z_1\bar{w}_1} 
 =\frac{E^2(z,w) - A^2(z,w)}{1-z_2\bar{w}_2}
 \]
 \[
 G^2(z,w):= \frac{A^2(z,w) - F^2(z,w)}{1-z_2\bar{w}_2} 
 =\frac{E^1(z,w) - A^1(z,w)}{1-z_1\bar{w}_1}.
 \]

\begin{theorem} \label{thm:maxmin}
 Let $\phi:\D^2 \to \mathcal{B}_1$ be inner and let $(A^1,A^2)$ be
 Agler kernels of $\phi$. Then, $G^1$ and $G^2$ as above are both
 positive semidefinite and $G = G^1 + G^2$.
 
 Conversely, suppose $G^1,G^2$ are positive semidefinite kernels satisfying
 $G=G^1+G^2$, and for $j=1,2$
 \[
 A^j(z,w) := F^j(z,w)+(1-z_j\bar{w}_j)G^j(z,w) 
 \]
 is positive semidefinite.  Then, $(A^1,A^2)$ are Agler kernels of $\phi$.
 \end{theorem}

The theorem says that in a certain sense $E^j$ dominates any possible
$A^j$, while $F^j$ is dominated by any possible $A^j$.

One can also study the relationship between the boundary regularity of
$\phi$ and the boundary regularity of functions in associated Hilbert
spaces.  In one variable, $\phi$ extends to be analytic at a boundary
point if and only if every element of $H^2\ominus \phi H^2$ does
\cite{hH64}.

In \cite{AC69}, Ahern and Clark studied the relationship between
regularity of an inner function $\phi$ on the boundary of the polydisk
and regularity of elements of $\Hphi$.  In particular, if every
element of $\Hphi$ extends holomorphically to a point $z \in \partial
\D^n$ with $|z_k|=1$ for some $k$, then $\phi$ depends on the $k$-th
variable alone.  This suggests $\Hphi$ is too big to be of use in
questions of regularity. The space $\Kphi$ is not quite correct
either, because it can be trivial even for rational inner functions.

For the remaining results, we restrict to finite-dimensional matrix
valued inner functions. Define
\[
\E = \C\setminus \cD,
\]
and then $\E^2$ will be what we call the exterior bidisk.  If $z =
(z_1,z_2) \in \C^2$, we sometimes write $1/\bar{z} =
(1/\bar{z}_1,1/\bar{z}_2)$ for short.

\begin{theorem} \label{thm:regintro} 
Let $\phi:\D^2 \to \mathcal{B}_1$ be a matrix valued inner function
(i.e. finite dimensional matrix valued). Let $X$ be an open subset of
$\T^2$ and let
\[
\begin{aligned}
X_1 &= \{ x_1 \in \T: \exists \ x_2 \in \T \text{ with }  (x_1,x_2) \in X \} \\
X_2 &= \{ x_2 \in \T: \exists \ x_1 \in \T \text{ with }  (x_1,x_2) \in X \}\\
S & = \{1/\bar{z}: \det \phi(z) = 0\}.
\end{aligned}
\]
Then the following are equivalent:   
\begin{itemize}
\item[$(i)$] The function $\phi$ extends continuously to $X.$
\item[$(ii)$] For some pair $(A^1,A^2)$ of Agler kernels of $\phi$,
 the elements of $\mcH(A^1)$, $\mcH(A^2)$ extend continuously to $X.$ 
\item[$(iii)$] There is a domain $\Omega$ containing 
\[ \D^2 \cup X \cup
  (X_1 \times \D) \cup (\D \times X_2) \cup (\E^2\setminus S) 
\] 
  on which $\phi$ and the elements of $\Kphi, \mcK^1_{\phi},
  \mcK^2_{\phi}$ extend to be analytic (and meromorphic on $\Omega
  \cup S$). Point evaluation in $\Omega$ is bounded in these spaces,
  and therefore, the kernels $G, F^1, F^2, E^1, E^2$ and all Agler
  kernels $(A^1,A^2)$ of $\phi$ extend to be sesqui-analytic on
  $\Omega\times \Omega.$
\end{itemize}
\end{theorem}

Continuity up to the boundary is a strong requirement.  For weaker
notions of extension to the boundary and their relations to Agler
decompositions, see \cite{AMY11} and \cite{AMY11b} by Agler,
McCarthy, and Young.

Restriction maps of canonical spaces are not only bounded, they are
also isometries into one variable spaces.

\begin{theorem} \label{thm:restrict} 
Let $\phi:\D^2 \to \mathcal{B}_1$ be a matrix valued inner function.
For almost every $t \in \T$, the map
\[
f \mapsto f(t,\cdot)
\]
embeds $\Kphi^1 \ominus Z_1 \Kphi$ and $\Kphi^1 \ominus \Kphi$
isometrically into $H^2(\T) \ominus \phi(t,\cdot) H^2(\T)$.
\end{theorem}

Finally, we can give an explicit description of the spaces involved
for a rational inner function $\phi = Q/p$, where $Q$ is an $N\times
N$ matrix polynomial and $p \in \C[z_1,z_2]$ has no zeros in $\D^2$.
If $Q$ has degree $d=(d_1,d_2)$, define
\[
\begin{aligned}
\Pzp &= \{ q/p \in H^2: \deg q \leq (d_1-1,d_2-1) \} \\
\Pop &= \{ q/p \in H^2: \deg q \leq (d_1,d_2-1) \} \\
\Ptp &= \{ q/p \in H^2: \deg q \leq (d_1-1,d_2) \}. 
\end{aligned}
\]
The requirement that $q/p \in L^2$ is analytic in the sense that the
zeros of $q$ must counter those of $p$.  Also, define $\tilde{Q}(z) =
z^d Q(1/\bar{z})^*$.  It turns out $\tilde{\phi} = \tilde{Q}/p$ is a
rational inner function, and we can give the following
analytic/algebraic description of the spaces associated to $\phi$.

\begin{theorem} \label{thm:rifdescrip}
 If $\phi = Q/p$ is a rational matrix inner function, then
\[
\begin{aligned}
\Kphi &= \{f \in \Pzp: \tilde{\phi} f \in \Pzp \} \\ \Kphi^1 &=
\{f \in \Pop: \tilde{\phi} f \in \Pop \} \\ \Kphi^2 &= \{f \in
\Ptp: \tilde{\phi} f \in \Ptp \}. \\
\end{aligned}
\]
\end{theorem}
 
Notice in particular that these spaces are all finite dimensional,
something shown already in Ball-Sadosky-Vinnikov \cite{BSV05} when
$\phi$ is regular up to $\T^2$. The significance here is that we have
given a complete description of the spaces involved even when there
are singularities on $\T^2$ (already done in the scalar case in
\cite{gK10}), and in addition, some of the dimensions involved can be
determined in a straightforward way.  Write $\det \phi =
\frac{\tilde{g}}{g},$ where $g$ is a polynomial with no zeros on
$\D^2$ and no factors in common with $\tilde{g}$.

\begin{theorem} \label{thm:rifdim}
With the same setup as the previous theorem,
\[
\begin{aligned}
\dim \Kphi^1 \ominus Z_1 \Kphi &= \dim \Kphi^1 \ominus \Kphi = \deg_2 \tilde{g} \\
\dim \Kphi^2 \ominus Z_2 \Kphi &= \dim \Kphi^2 \ominus \Kphi = \deg_1 \tilde{g}. 
\end{aligned}
\]
\end{theorem}

The dimension of $\Kphi$ depends on the nature of the zeros of $p$ on
$\T^2$ and properties of $Q$ and therefore cannot be given by such a
simple formula. See Example \ref{ex:reguniquedecomp}.

It is then possible to show that the canonical Agler kernels
$(E^1,F^2)$ or $(F^1,E^2)$ have minimal dimensions.

\begin{corollary} \label{cor:minag}
Let $\phi$ be an $N\times N$ matrix valued rational inner function on
$\D^2$. Write $\det \phi =\tilde{g}/g$ in lowest terms and let $d=
(d_1,d_2)=\deg \tilde{g}$. If $(A^1,A^2)$ are Agler kernels, then 
\[
\dim \mcH(A^1) \geq d_2 \text{ and } \dim \mcH(A^2) \geq d_1.
\]
\end{corollary}

\begin{remark}
We point out some known applications of the above work in the scalar
case of $\phi = \tilde{p}/p$.

The existence of Agler kernels with minimal dimensions was important
in proving one of the main results in Agler-McCarthy-Young
\cite{AMY12}.

The details of the canonical spaces and kernels have been important in
understanding function theory on \emph{distinguished varieties} in
\cite{gK10dv}, \cite{JKM12}.  Distinguished varieties are algebraic
curves in $\C^2$ which exit the bidisk through the distinguished
boundary \cite{AM05}.

Since the spaces $\Kphi, \Kphi^1,\Kphi^2$ consist of rational
functions with denominator $p$, orthogonality relations in these
spaces can be reinterpreted as orthogonality relations between spaces
of polynomials using the inner product of $L^2(\frac{1}{|p|^2}
d\sigma, \T^2).$  In particular, the fact that
\[
\Kphi^1\ominus \Kphi \perp \Kphi^2
\]
can be reinterpreted as the key condition required to characterize
measures of the form $\frac{1}{|p|^2} d\sigma$ on $\T^2$ in
Gernonimo-Woerdeman \cite{GW04}.  This condition can further be used
to give a characterization of positive two variable trigonometric
polynomials $t$ which can written as $|p(z,w)|^2$, where $p$ is a
polynomial with no zeros on $\overline{\D^2}$.
\end{remark}

The following result of Kummert \cite{aK89} on transfer function
representations is a direct consequence of Theorem \ref{thm:rifdim}
(see also Ball-Sadosky-Vinnikov \cite{BSV05}).

\begin{corollary} \label{cor:transfer}
Let $\phi$ be an $N\times N$ matrix valued rational inner function on
$\D^2$. Write $\det \phi = \tilde{g}/g$ in lowest terms and let
$d=(d_1,d_2) = \deg \tilde{g}$, $|d|=d_1+d_2$.  Then, there exists an
$(N+|d|)\times (N+|d|)$ unitary matrix $U,$ which we write in block
form
\[
U = \begin{matrix} \ & \begin{matrix} \C^N & \C^{|d|} \end{matrix}
  \\
\begin{matrix} \C^N \\ \C^{|d|} \end{matrix} & \begin{pmatrix} A &
  B \\ C & D \end{pmatrix} \end{matrix}
=  \begin{matrix} \ & \begin{matrix} \C^N & \C^{d_1} & \C^{d_2} \end{matrix}
  \\
\begin{matrix} \C^N \\ \C^{d_1}\\ \C^{d_2} \end{matrix} & \begin{pmatrix} A &
  B_1 & B_2 \\ C_1 & D_{11} & D_{12} \\ C_2 & D_{21} & D_{22} \end{pmatrix} \end{matrix}
\]
such that
\begin{equation} \label{transferform}
\phi(z) = A + B d(z)(I-D d(z))^{-1} C,
\end{equation}
where $d(z) = \begin{pmatrix} z_1 I_{d_1} & 0 \\ 0 & z_2
  I_{d_2} \end{pmatrix}$.

Furthermore, $(N+|d|)\times (N+|d|)$ is the minimum possible size of
such a representation.
\end{corollary}

It is a standard calculation that given a unitary $U$,
\eqref{transferform} yields a matrix rational inner function, so the
transfer function representation gives a way to write down every
rational inner function, although the representation may not be
unique.

We have emphasized Agler decompositions with minimal dimensions
because of the following application of Theorem \ref{thm:rifdim} to
the study of \emph{three} variable rational inner functions.  We offer
an improvement to a result of Knese \cite{gKsacrifott}, which in turn
was a generalization of a result of Kummert \cite{aK89b}.

\begin{theorem} \label{thm:threevar} Let $p \in \C[z_1,z_2,z_3]$ have degree $(n,1,1)$ and no
  zeros on $\overline{\D^3}$, let $\tilde{p}(z) =
  z_1^nz_2z_3\overline{p(1/\bar{z})}$, and define the rational inner
  function $\phi = \tilde{p}/p$.Then, $\phi$ has an Agler decomposition of the form
\[
1 - \phi(z)\overline{\phi(w)} = \sum_{j=1}^{3} (1-z_j\bar{w}_j) SOS_j(z,w),
\]
where $SOS_2$ and $SOS_3$ are sums of two squares, while $SOS_1$ is a
sum of $2n$ squares.
\end{theorem}

To be clear, $SOS_2(z,z) = \frac{1}{|p(z)|^2}(|A_1(z)|^2+|A_2(z)|^2)$
for some polynomials $A_1,A_2\in \C[z_1,z_2,z_3]$, and similarly for
$SOS_3$, while $SOS_1(z,z) = \frac{1}{|p(z)|^2}
\sum_{j=1}^{2n}|B_j(z)|^2$ for some polynomials $B_1,\dots, B_{2n} \in
\C[z_1,z_2,z_3]$.  Part of the significance of the result is that no
such decomposition can generally exist for rational inner functions on
$\D^3$ (regardless of the bounds on the number of squares involved).
For more on this, see \cite{gKsacrifott}.

On the other hand, part of the significance of the result \emph{is}
the bounds obtained on the number of squares involved.  The original
theorem in \cite{gKsacrifott} had the non-optimal bounds of
$4n$\footnote{Regrettably, due to an arithmetic error $4(n-1)$ was
  written in the original paper \cite{gKsacrifott} instead of $4n$.
}, $2(n+1), 2$ for the number of squares in $SOS_1,SOS_2,SOS_3$.

In \cite{gKrifitsac}, an explicit Agler decomposition was found for
$\tilde{p}/p$ when $p(z_1,z_2,z_3) = 3- z_1-z_2-z_3$ with sums of
squares terms $SOS_1,$ $SOS_2,$ $SOS_3$ containing 3 squares each.  Although
$p$ has a zero on $\T^3$, the proof for Theorem \ref{thm:threevar}
works for $p$ (and in fact should work for any $p$ with no zeros on
$\D^3$ when $p$ and $\tilde{p}$ have no factors in common).  We
conclude that $\tilde{p}/p$ has an Agler decomposition with $2$
squares in each sums of squares term. It is shown in \cite{gKrifitsac}
that none of the terms $SOS_1,SOS_2, SOS_3$ can be written as a single
square. This shows Theorem \ref{thm:threevar} is optimal when $n=1$.  

\begin{question} \label{q:optimal} 
Is Theorem \ref{thm:threevar} optimal for $n>1$?  Namely, is there a
rational inner function of degree $(n,1,1)$ ($n>1$) such that for
every Agler decomposition, $SOS_1$ is a sum of $2n$ or more squares?
\end{question}

The rest of the paper is summarized in the table of contents.

\tableofcontents

\section{Theorem \ref{thm:fundagler} on fundamental Agler decompositions}

In this section, we sketch the proof of Theorem \ref{thm:fundagler}.

We first note some simple inclusions 
\[
\Kphi, Z_j \Kphi \subset \Kphi^j \subset \Hphi^j \subset \Hphi.
\]
The spaces $\Kphi, \Kphi^j$ should be thought of as ``small'' since
they are finite dimensional in the case of rational $\phi$.  

The following can be proved straight from definitions.

\begin{prop} \label{prop:max}
For $j=1,2$, the space $\Hphi^j$ is invariant under multiplication by
$Z_j,$ and even more, if $f \in \Hphi$ and if $Z_j^k f \in \Hphi$ for
all $k>0$, then $f \in \Hphi^j$.
\end{prop}

Hence, $\Hphi^j$ is maximal among all $Z_j$ invariant subspaces of
$\Hphi$.  As $\phi$ is unitary valued a.e. on $\T^2$, it follows that $\phi L^2_{\bullet -} = (\phi
L^2_{\bullet +})^{\perp}$.  An observation in \cite{kB11} is
that
\[
\Hphi^1 = H^2 \cap (\phi L^2_{\bullet +})^{\perp} = H^2\cap (\phi
H^2)^{\perp} \cap (\phi L^2_{-+})^{\perp} = \Hphi \cap (\phi
L^2_{-+})^{\perp},
\]
so that we have the following: 

\begin{prop} \label{prop:bickel}
\[
\Hphi^1 = \Hphi \ominus P_{\Hphi}(\phi L^2_{-+})
\]
and so
\[
\Hphi\ominus \Hphi^1 = \overline{P_{\Hphi}(\phi L^2_{-+})}
= \overline{P_{H^2}(\phi L^2_{-+})} =
\overline{P_{L^2_{+\bullet}} (\phi L^2_{-+})}.
\]
\end{prop}

Here $P$ denotes orthogonal projection onto the space in the subscript
and the overset bars denote closures. The second equality follows from
the fact that $\phi H^2 \perp \phi L^2_{-+}$. The last equality
follows from the fact that $\phi L^2_{-+} \subset L^2_{\bullet +}$ and
therefore projecting onto either $H^2$ or $L^2_{+\bullet}$ has the
same effect.

It is easy to show that $P_{L^2_{+\bullet}}(\phi L^2_{-+})$ is
invariant under multiplication by $Z_2$ since $\phi$ is in
$H^\infty$. So, $\Hphi\ominus \Hphi^1$ is invariant under
multiplication by $Z_2$, and therefore $\Hphi \ominus \Hphi^1 \subset
\Hphi^2$ by Proposition \ref{prop:max}.  By Lemma \ref{lem:orth} below
and the fact $\Hphi^1\cap \Hphi^2 = \Kphi$, we immediately have the
following two propositions.

\begin{prop} \label{prop:orth}
\[
\Hphi \ominus \Hphi^1 = \Hphi^2 \ominus \Kphi
\]
so that
\[
\Hphi = (\Hphi^1\ominus \Kphi) \oplus (\Hphi^2 \ominus \Kphi)
\oplus \Kphi.
\]
\end{prop}

\begin{prop} \label{thm:shift}
For $j=1,2$, $\Hphi^j \ominus \Kphi$ is invariant under multiplication
by $Z_j$.
\end{prop}

The following is a standard fact.

\begin{prop} \label{prop:rk} The space $\mcH_\phi$ is a reproducing kernel Hilbert space with
  reproducing kernel
\[
\frac{1 - \phi(z)
  \phi(w)^* }{(1-z_1\bar{w}_1)(1-z_2\bar{w}_2)}
\]
for $z,w \in \D^2$.
\end{prop}

\begin{proof}[Sketch of Proof of Theorem \ref{thm:fundagler}]
Multiplication by $Z_j$ is a pure isometry on $H^2$ and hence on
$\Hphi^j$ and $\Hphi^j\ominus \Kphi$. By Propositions \ref{prop:max}
and \ref{thm:shift}, we have the orthogonal decompositions
\[
\Hphi^j = \bigoplus_{k\geq 0} Z_j^k (\Hphi^j \ominus Z_j \Hphi^j)
\text{ and } \Hphi^j\ominus \Kphi = \bigoplus_{k\geq 0} Z_j^k ((\Hphi^j\ominus
\Kphi)\ominus Z_j(\Hphi^j\ominus \Kphi)).
\]
Because of this, the reproducing kernels for $\Hphi^j$ and
$\Hphi^j\ominus K_{\phi}$ are given by
\[
\frac{E^j(z,w)}{1-z_j\bar{w}_j} \text{ and } \frac{F^j(z,w)}{1-z_j\bar{w}_j}
\]
respectively (recall Notation \ref{efgnotation}).  By Proposition
\ref{prop:orth},
\[
\Hphi = \Hphi^1 \oplus (\Hphi^2\ominus \Kphi) 
=
(\Hphi^1\ominus \Kphi) \oplus (\Hphi^2\ominus \Kphi) \oplus
\Kphi.
\]

Therefore, the reproducing kernel for $\Hphi$ can be
decomposed in the following two ways: 
\[
\begin{aligned}
\frac{1 - \phi(z)
  \phi(w)^*}{(1-z_1\bar{w}_1)(1-z_2\bar{w}_2)} &=
\frac{E^1(z,w)}{1-z_1\bar{w}_1} + \frac{F^2(z,w)}{1-z_2\bar{w}_2} \\
&= \frac{F^1(z,w)}{1-z_1\bar{w}_1} + \frac{F^2(z,w)}{1-z_2\bar{w}_2} +
G(z,w).
\end{aligned}
\]

\end{proof}

The following general Hilbert space lemma was used above.  It makes a
few arguments later easier to digest.

\begin{lemma} \label{lem:orth}
Let $\mcH$ be a Hilbert space with (closed) subspaces
$\mcK_1,\mcK_2$. If $\mcH \ominus \mcK_1 \subset \mcK_2$, then
\[
\mcH\ominus \mcK_1 = \mcK_2 \ominus (\mcK_1\cap \mcK_2).
\]
\end{lemma}

\begin{proof}
The inclusion $\subset$ is trivial.  For the opposite direction,
suppose $f \in \mcK_2\ominus (\mcK_1\cap \mcK_2)$ and $f \perp \mcH
\ominus \mcK_1$.  Then, $f \in \mcK_1$ and $f \in \mcK_2$ making $f$
orthogonal to itself.  Hence, $f=0$.
\end{proof}

 \section{Two general lemmas on reproducing kernels}

 We record two standard facts about reproducing kernel Hilbert
 spaces. See \cite{nA50} or \cite{AM02} for more general information.  

Let $B(\mcV)$ be the bounded linear operators on $\mcV$, our separable
Hilbert space with inner product $\ip{\cdot}{\cdot}_{\mcV}$.  Given a
function $H: X\times X \to B(\mcV)$, the notation
 \[
 H \kgeq 0
 \]
 means $H$ is a positive semidefinite kernel. Namely, for any
 $x_1,\dots, x_N \in X$, the block operator on $\mcV^N$ 
 \[
 (H(x_j,x_k))_{jk}
 \]
 is positive semidefinite.  In addition, for $K:X\times X \to
 B(\mcV)$, the notation $H\kgeq K$ means $H - K \kgeq 0$.  We write
 $\mathcal{H}(K)$ for the canonical reproducing kernel Hilbert space
 of $\mcV$ valued functions on $X$ associated to a positive
 semidefinite kernel $K,$ and $\ip{\cdot}{\cdot}_{\mcH(K)}$ denotes the
 inner product in $\mcH(K)$.  The reproducing kernels are $K_x(\cdot)v
 = K(\cdot,x)v$ for $x \in X, v\in \mcV$, by which we mean
\[
\ip{f}{K_x v}_{\mcH(K)} = \ip{f(x)}{v}_{\mcV}
\]
for any $f\in \mcH(K)$.  The span of these functions is dense in
$\mcH(K)$.

\begin{remark} The setting of
  vector valued functions and operator valued kernels can easily be
  reduced to the setting found in standard references of scalar
  functions and kernels by viewing our ``points'' as elements of
  $X\times \mcV$ as opposed to $X$.  Evaluating $f \in \mcH(K)$ at the
  ``point'' $(x,v)$ would refer to $\ip{f(x)}{v}_{\mcV}$.
\end{remark}

It is a standard fact that a function $f:X \to \mcV$ is in
$\mcH(K)$ if and only if there is an $\alpha \geq 0$ such that
\[
\alpha K(y,x) \kgeq f(y)f(x)^*,
\]
where for any $v \in \mcV$, $v^*$ denotes the functional
$\ip{\cdot}{v}_{\mcV}$.  The minimum of all such $\alpha$'s is
$\|f\|_{\mcH(K)}^{2}$. So, if $H \kgeq K \kgeq 0$ and if $f\ne
0,$ then
\[
H(y,x) \kgeq K(y,x) \kgeq \frac{f(y)f(x)^*}{\|f\|^2_{\mcH(K)}}.
\]
This implies $\|f\|_{\mcH(H)} \leq \|f\|_{\mcH(K)},$ which gives the
following lemma:

 \begin{lemma} \label{lem:poskernels}
If $H \kgeq K \kgeq 0$ on $X$, then $\mcH(K) \subseteq \mcH(H),$ and the
embedding $\iota: \mcH(K) \to \mcH(H)$ is a contraction.
 \end{lemma}

If $F\subset X$ is a finite set and $v:F \to \mcV$ is a function, then
\[
\|\sum_{x \in F}  K_x v(x) \|_{\mcH(K)}^2 = \sum_{x,y \in F}
\ip{K(y,x)v(x)}{v(y)}_{\mcV}
\]
essentially by definition of the inner product in $\mcH(K)$.  Set $f =
\sum_{x \in F} K_x v(x)$.  So, if $H \kgeq K \kgeq 0$, then  $f \in 
\mcH(H)$ and $\|f\|_{\mcH(H)}^2 \leq \|f\|_{\mcH(K)}^2$,
i.e.
\[
\sum_{x,y \in F} \ip{K_xv(x)}{K_yv(y)}_{\mcH(H)} \leq \sum_{x,y \in F}
\ip{K(y,x)v(x)}{v(y)}_{\mcV}.
\]

We also need the following:

 \begin{lemma} \label{lem:RKHS}
Let $H$ be a
   positive semidefinite kernel on $X,$ and let $\mcK$ be a closed
   subspace of $\mcH(H)$ with reproducing kernel $K$.  Suppose $H
   \kgeq L \kgeq 0$ and $L_x(\cdot)v = L(\cdot,x)v \in \mcK$ for all $x
   \in X, v\in \mcV$.  Then, $K \kgeq L$.
 \end{lemma}

\begin{proof}
Let $F \subset X$ be a finite set and $v:F \to \mcV$ a function. Define
$f = \sum_{x \in F} L_x v(x)$, $g = \sum_{x \in F} K_xv(x)$.
Then, $f,g \in \mcK$ and $f \in \mcH(L)$.  Therefore,
\[
\begin{aligned}
\| f\|^2_{\mcH(L)} &= \sum_{x,y \in F} \ip{L(y,x)v(x)}{v(y)}_{\mcV}\\
&= \ip{\sum_{x \in F} L_xv(x)}{\sum_{y \in F} K_yv(y)}_{\mcH(H)} \\
&= \ip{f}{g}_{\mcH(H)} \\
&\leq \|f\|_{\mcH(H)} \|g\|_{\mcH(H)} \\
&\leq \|f\|_{\mcH(L)} \|g\|_{\mcH(H)} \text{ by Lemma
  \ref{lem:poskernels}, } 
\end{aligned}
\]
which shows $\|f\|^2_{\mcH(L)} \leq \|g\|^2_{\mcH(H)}$.  Expanding this
out shows
\[
\sum_{x,y \in F} \ip{L(y,x)v(x)}{v(y)}_{\mcV} \leq \sum_{x,y \in F} \ip{K(y,x)v(x)}{v(y)}_{\mcV},
\]
which shows $L \kleq K$.
\end{proof}

 \section{Theorem \ref{thm:maxmin} on maximality and minimality}

 The ``canonical'' Agler decompositions from Theorem
 \ref{thm:fundagler} are maximal and minimal in the sense described in
 the following theorem. Moreover, all other Agler decompositions can
 be characterized in terms of properties of $F^1,F^2,G$.

 This maximality and minimality property is found in Theorem 5.5 of
 \cite{BSV05}.  The following result is more general in one
 sense; we consider Agler decompositions which do not necessarily come
 from an orthogonal decomposition inside $H^2$. 

This generality is nontrivial. Specifically, by considering monomial inner 
functions like $\phi(z)=z_1^2z_2$, one can show there are Agler decompositions
that cannot  be written as convex combinations of Agler decompositions
coming from orthogonal decompositions inside $H^2.$ See Example
\ref{nonextremeaglerdecomps}.

\begingroup
\def\thetheorem{\ref{thm:maxmin}}
 \begin{theorem} 
 Let $\phi:\D^2 \to \mathcal{B}_1$ be inner and let 
 $(A^1,A^2)$ be Agler kernels of $\phi$. Then, for $j,k \in \{1,2\}$ distinct 
 \[
 G^j(z,w) := \frac{A^j(z,w) - F^j(z,w)}{1-z_j\bar{w}_j} 
 =\frac{E^k(z,w) - A^k(z,w)}{1-z_k\bar{w}_k}
 \]
 is positive semidefinite and
 \[
 G(z,w) = G^1(z,w)+G^2(z,w).
 \]
 Conversely, suppose $G^1,G^2$ are positive semidefinite and satisfy
 $G=G^1+G^2$, while for $j=1,2,$
 \[
 A^j(z,w) := F^j(z,w)+(1-z_j\bar{w}_j)G^j(z,w) 
 \]
 is positive semidefinite.  Then, $(A^1,A^2)$ are Agler kernels of $\phi$.
 \end{theorem}
\addtocounter{theorem}{-1}
\endgroup

 For a quick corollary, observe that if $G=0$, or equivalently, $\Kphi
 = \{0\}$, then $\phi$ has a unique Agler decomposition.  On the other
 hand, if $\phi$ has a unique Agler decomposition, then $E^2=F^2$,
 $E^1=F^1$, and then $G=0$, yielding the following:

 \begin{corollary} \label{cor:unique}
 An inner function $\phi$ has a unique Agler decomposition if and only
 if $\Kphi = \{0\}.$
 \end{corollary}

 \begin{proof}[Proof of Theorem]
 Set $L(z,w) = A^1(z,w)/(1-z_1\bar{w}_1).$ Since
 \[
 \frac{1-\phi(z)\phi(w)^*}{(1-z_1\bar{w}_1)(1-z_2\bar{w}_2)} =
   \frac{A^1(z,w)}{1-z_1\bar{w}_1} + \frac{A^2(z,w)}{1-z_2\bar{w}_2} \kgeq
   L(z,w),
 \]
 Lemma \ref{lem:poskernels} implies that $\mcH(L) \subseteq \Hphi.$ In addition, 
 \[
 (1-z_1\bar{w}_1)L(z,w) = A^1(z,w) \kgeq 0
 \]
 shows $\mcH(L)$ is invariant under multiplication by $Z_1.$ In particular, 
$Z^j_1L_{w}v \in \mcH(L) \subseteq \Hphi$ for all $j\geq 0$, $w \in
 \D^2, v \in \mcV$.  Then, by Proposition \ref{prop:max}, each $L_wv \in \Hphi^1.$ 
It follows from Lemma \ref{lem:RKHS} that
 \[
 \frac{E^1(z,w)}{1-z_1\bar{w}_1} \kgeq \frac{A^1(z,w)}{1-z_1\bar{w}_1},
 \]
and so $G^2(z,w) \kgeq 0.$ The remainder of the forward implication follows 
from algebraic manipulations.  

For the converse, we immediately have
 \[
 \begin{aligned}
 & (1-z_1\bar{w}_1)A^2(z,w) + (1-z_2\bar{w}_2)A^1(z,w) \\
 =& (1-z_1\bar{w}_1)(F^2(z,w)+(1-z_2\bar{w}_2)G^2(z,w)) \\
 &+ (1-z_2\bar{w}_2)(F^1(z,w)+(1-z_1\bar{w}_1)G^1(z,w)) \\
 =& (1-z_1\bar{w}_1)F^2(z,w)+(1-z_2\bar{w}_2)F^1(z,w) \\
 &+ (1-z_1\bar{w}_1)(1-z_2\bar{w}_2)G(z,w) \\
= & 1-\phi(z)\phi(w)^*,
 \end{aligned}
 \]
 so that $(A^1,A^2)$ are Agler kernels of $\phi.$
 \end{proof}

\section{More details on the canonical subspaces}
The previous sections show that the study of an inner function in two
variables hinges on the subspaces:
\[
(\Hphi^j\ominus \Kphi) \ominus Z_j (\Hphi^j\ominus \Kphi) \text{ for } j=1,2
\]
\[
\Hphi^j \ominus Z_j \Hphi^j \text{ for } j=1,2
\]
and $\Kphi$.  The main point of this section is that all of these
subspaces are ``small'' in the sense that they sit inside either
$\Kphi^1$ or $\Kphi^2$. 
Let 
\[
\begin{aligned}
L^2_{0-} &= \{f \in L^2: \supp(\hat{f}) \subset \{0\} \times \Z_{-}\}\\
L^2_{0+} &= \{f \in L^2: \supp(\hat{f}) \subset \{0\} \times \Z_{+}\}
\end{aligned}
\]
and define $L^2_{-0}$ and $L^2_{+0}$ similarly.  We shall use $A \vee
B$ to denote the closed linear span of two sets $A$ and $B$ in a common
Hilbert space.

\begin{prop} \label{prop:sub1}
\[
(\Hphi^1\ominus \Kphi) \ominus Z_1(\Hphi^1\ominus \Kphi) =
\mcK^1_{\phi} \ominus \Kphi = \overline{P_{\Kphi^1} (\phi L^2_{0-})}
\]
\[
(\Hphi^2\ominus \Kphi) \ominus Z_2(\Hphi^2\ominus \Kphi) =
\mcK^2_{\phi} \ominus \Kphi = \overline{P_{\Kphi^2} (\phi L^2_{-0})}.
\]
\end{prop}

\begin{proof} 
 We define the two subspaces $\mcQ$ and $\mcR$:
\[
\begin{aligned}
\mcR &= L^2_{+\bullet} \cap \phi L^2_{--} \\
\mcQ &= L^2_{+\bullet} \cap \phi L^2_{-\bullet} = L^2_{+\bullet}
\ominus \phi L^2_{+\bullet}
\end{aligned}
\]
and calculate 

\[
\begin{aligned}
\mcQ \ominus \mcR &= P_{\mcQ} ( \mcR^{\perp} )&& \nonumber \\
& = P_{\mcQ} ( L^2_{- \bullet} \vee \phi(L^2_{-+} \oplus L^2_{+ \bullet}))&& \nonumber \\
& = \overline{ P_{\mcQ} (L^2_{- \bullet} + \phi(L^2_{-+} \oplus
L^2_{+ \bullet}) )} &&\nonumber \\
&= \overline{P_{\mcQ} ( \phi L^2_{-+})} && \nonumber \text{ since } \mcQ
\perp L^2_{-\bullet}, \phi L^2_{+\bullet} \\
&= \overline{(P_{\mcQ}+P_{\phi L^2_{+\bullet}})( \phi L^2_{-+})} 
&& \text{ since }  P_{\phi L^2_{+\bullet}}( \phi L^2_{-+}) = 0 \\
& =  \overline{P_{ L^2_{+\bullet} } ( \phi L^2_{-+})} &&  \nonumber \text{ since }
L^2_{+ \bullet}  = \mcQ \oplus \phi L^2_{+ \bullet}  \\ 
& = \Hphi^2 \ominus \Kphi  &&\text{ by Proposition \ref{prop:bickel}.} 
\end{aligned}
\]

Therefore, the ``wandering'' subspace satisfies
\[
\begin{aligned}
(\Hphi^2\ominus \Kphi) \ominus Z_2(\Hphi^2\ominus \Kphi) &=
(\mcQ \ominus \mcR) \ominus Z_2 (\mcQ \ominus \mcR) \\
&=(\mcQ \ominus \mcR)\ominus (\mcQ \ominus Z_2\mcR) \quad \text{ since }
  Z_2\mcQ = \mcQ \\
&=Z_2\mcR \ominus \mcR \subset \Hphi^2.
\end{aligned}
\]
Recall that $Z_2\mcR \cap \Hphi^2 = H^2 \cap Z_2\phi L^2_{--} =
\Kphi^2 .$ Using Lemma \ref{lem:orth}, we now intersect with $\Hphi^2$
to obtain
\[
\begin{aligned}
Z_2\mcR \ominus \mcR &= ( Z_2\mcR \cap \Hphi^2) \ominus ( \mcR \cap \Hphi^2) \\
&= \Kphi^2 \ominus \Kphi 
\end{aligned}
\]
equals the wandering subspace. We can also identify $\Kphi^2 \ominus \Kphi$ with 
a ``closure of a projection'' as follows:
\[
\begin{aligned}
\Kphi^2 \ominus \Kphi &= P_{\Kphi^2 } ((\Kphi)^{\perp})  \\
&=  P_{\Kphi^2 } ((L^2 \ominus H^2) \vee \phi(L^2_{+ \bullet} \oplus L^2_{-+})  )\\
&= \overline{ P_{\Kphi^2 } ( (L^2 \ominus H^2) + \phi(L^2_{+
\bullet} \oplus L^2_{-+}) )} \\
& =  \overline{ P_{\Kphi^2 }  (\phi L^2_{-0})},
\end{aligned}
\]
since $\Kphi^2 \subset Z_2 \phi L^2_{--} \perp
\phi(L^2_{+\bullet}\oplus Z_2 L^2_{-+}).$
\end{proof}

\begin{prop} \label{prop:sub2}
\[ 
\begin{aligned}
\Hphi^1 \ominus Z_1 \Hphi^1 = \mcK^1_{\phi} \ominus Z_1\Kphi &= \overline{ P_{\Kphi^1} (L^2_{0+})} \\
\Hphi^2 \ominus Z_2 \Hphi^2 = \mcK^2_{\phi} \ominus Z_2\Kphi &= \overline{ P_{\Kphi^2} (L^2_{+0})}.
\end{aligned}
\]
\end{prop}

\begin{proof}
Since 
\[
E^1(z,w)=F^1(z,w) + (1-z_1\bar{w}_1)G(z,w)
\]
and since $F^1_w v \in \Kphi^1$ for $w \in \D^2, v \in \mcV$, we see
that $E^1_wv \in \Kphi^1$. Hence, $\Hphi^1\ominus Z_1\Hphi^1 \subset
\Kphi^1$, so by Lemma \ref{lem:orth}
\[
\Hphi^1\ominus Z_1 \Hphi^1 = \Kphi^1 \ominus Z_1 \Kphi.
\]
We can also identify $\Kphi^1 \ominus Z_1 \Kphi$ with a ``closure of a
projection'' because
\[
\begin{aligned}
\Kphi^1 \ominus Z_1 \Kphi &=  P_{\Kphi^1}(( Z_1\Kphi)^{\perp})  \\
&=  P_{\Kphi^1}(  (L^2 \ominus Z_1H^2) 
\vee Z_1 \phi( L^2_{+ \bullet} \oplus L^2_{-+}) )\\
&=  \overline{ P_{\Kphi^1}( (L^2 \ominus Z_1H^2) 
+ Z_1 \phi( L^2_{+ \bullet} \oplus L^2_{-+}) )}\\
& =  \overline{ P_{\Kphi^1} (L^2_{0+})}.
\end{aligned}
\]
\end{proof}

The characterization in Theorem  \ref{thm:maxmin} implies that
the reproducing kernel Hilbert spaces associated to any Agler decomposition must also
sit inside either $\Kphi^1$ or $\Kphi^2.$

\begin{corollary} \label{cor:Agker} Let $(A^1,A^2)$ be Agler kernels of $\phi$. Then 
\[ \mcH(A^j) \text{ is contained contractively in } \Kphi^j \text{ for } j=1,2.
\]
\end{corollary}
\begin{proof} By Theorem \ref{thm:maxmin}, for $j=1,2,$ we can write 
\[ 
A^j(z,w) = F^j(z,w) + (1-z_j\bar{w}_j)G^j(z,w),
\]
where each $G^j$ is positive semidefinite and $G = G^1 +G^2.$  From Proposition 
\ref{prop:sub1} and the definitions of $F^j$ and $G$, it is clear that 
$\Kphi^j$ has reproducing kernel $F^j +G.$ As
\[ 
F^1 + G - A^1 =  G^2 + z_1\bar{w}_1G^1 \kgeq 0,
\]
it follows from Lemma \ref{lem:poskernels}  that $\mcH(A^1)$ is contained
contractively in $\mcH(F^j +G)= \Kphi^1$ and similarly, $\mcH(A^2)$ is
in $\Kphi^2.$
\end{proof}

\section{Matrix inner functions and Theorem \ref{thm:regintro} on regularity}
For the rest of the paper we assume $\mcV = \C^N$ and therefore, that
$\phi$ is an $N\times N$ matrix valued inner function on $\D^2$.
Define
\[
\E = \C\setminus \cD,
\]
and then $\E^2$ will be what we call the exterior bidisk.

We now restate and prove Theorem \ref{thm:regintro}.
\begingroup
\def\thetheorem{\ref{thm:regintro}}
\begin{theorem}  
Let $\phi:\D^2 \to \mathcal{B}_1$ be a matrix valued inner function. 
Let $X$ be an open subset of $\T^2$ and let
\[
\begin{aligned}
X_1 &= \{ x_1 \in \T: \exists \ x_2 \in \T \text{ with }  (x_1,x_2) \in X \} \\
X_2 &= \{ x_2 \in \T: \exists \ x_1 \in \T \text{ with }  (x_1,x_2) \in X \}\\
S & = \{1/\bar{z}: \det \phi(z) = 0\}.
\end{aligned}
\]
Then the following are equivalent:   
\begin{itemize}
\item[$(i)$] The function $\phi$ extends continuously to $X.$
\item[$(ii)$] For some pair $(A^1,A^2)$ of Agler kernels of $\phi$,
 the elements of $\mcH(A^1)$, $\mcH(A^2)$ extend continuously to $X.$ 
\item[$(iii)$] There is a domain $\Omega$ containing
 \[ \D^2 \cup X \cup
  (X_1 \times \D) \cup (\D \times X_2) \cup (\E^2\setminus S)
\] 
on which $\phi$ and the elements of $\Kphi, \mcK^1_{\phi},
  \mcK^2_{\phi}$ extend to be analytic (and meromorphic on $\Omega
  \cup S$). Point evaluation in $\Omega$ is bounded in these spaces,
  and therefore, the kernels $G, F^1, F^2, E^1, E^2$ and all Agler
  kernels $(A^1,A^2)$ of $\phi$ extend to be sesqui-analytic on
  $\Omega\times \Omega$.
\end{itemize}
\end{theorem}
\addtocounter{theorem}{-1}
\endgroup

We will prove $(i) \Rightarrow (iii) \Rightarrow (ii) \Rightarrow (i).$ As most analysis
lies in proving $(i) \Rightarrow (iii)$, we consider that implication first.

\begin{claim} $\phi$ extends to be analytic in some domain
  $\Omega.$ \end{claim} 
\begin{proof}
Suppose that $X$ is an open subset of $\T^2$ and $\phi$ extends to be
continuous on $ \D^2 \cup X$.  Then, $\phi$ is invertible in a
neighborhood $W^+ \subset \D^2$ with $X \subset \text{closure}(W^{+})$
(since $\phi$ is a unitary almost everywhere on $\T^2$).  The
following
\begin{equation} \label{phiextend}
\phi(z) = (\phi(1/\bar{z})^*)^{-1}
\end{equation}
gives a definition of $\phi$ on $W^{-} := \{1/\bar{z}: z \in
W^{+}\}$.
The extended $\phi$ is holomorphic on $W^{+}\cup W^{-}$ and continuous
on $W^{+} \cup X \cup W^{-}$ since $\phi$ is unitary valued on $X$.
By the continuous edge-of-the-wedge theorem (Theorem A of Rudin
\cite{wR71}), there is a domain $\Omega_0$ containing $W^{+}\cup X
\cup W^{-}$, which depends only on $X, W^{\pm}$, on which $\phi$
extends to be holomorphic.  Moreover, $\phi$ is already holomorphic on
$\D^2$, meromorphic in $\E^2$, and holomorphic away from the set $S$
using the definition \eqref{phiextend}.

We can extend this domain further using a result in Rudin \cite{wR69}
(Theorem 4.9.1, which we provide as Proposition \ref{prop:rudin2}
below).  It says, roughly, that if a holomorphic function $f$ on
$\D^2$ extends analytically to a neighborhood $N_x$ of some
$x=(x_1,x_2) \in \T^2,$ then $f$ extends analytically to an open set
containing $\{x_1\} \times \D$ and $\D \times \{x_2\}.$ As the
edge-of-the-wedge theorem guarantees $\phi$ extends to a neighborhood
$N_x$ of each $x \in X$, Proposition \ref{prop:rudin2} implies $\phi$
extends analytically to an open set $U$ containing $(X_1 \times \D )
\cup (\D \times X_2),$ and the open set depends only on the
$\{N_x\}_{x \in X}.$ This detail is contained in the \emph{proof} of
Proposition \ref{prop:rudin2}.
\end{proof}
\begin{claim} Elements of $\Kphi, \Kphi^1, \Kphi^2$ are analytic in
  $\Omega.$ \end{claim}
\begin{proof}
Consider now $f \in \Kphi$; the proof is similar for the other
subspaces. Since $\phi^* f \in L^2_{--}$, we may write $f = \phi
\overline{Z_1 Z_2 g}$ for some $g \in H^2$.  This allows us to define
$f$ analytically outside of $\D^2$ as follows:
\[
f(z) = \frac{1}{z_1z_2} \phi(z) \overline{g(1/\bar{z})}
\]
for $z \in \E^2\setminus S$.  Note $f$ is meromorphic in $\E^2$.  With
this definition, for any compact subset $X_0 \subset X$ and $z \in
X_0$ 
\[
\lim_{r \searrow 1} f(rz) = \phi(z) \bar{z}_1\bar{z}_2 \overline{g(z)} = f(z)
\]
in $L^2(X_0)$ since as $r \searrow 1$
\[
g(1/r\bar{z}) \to g(z) 
\]
in $L^2(\T^2)$, while $\phi(rz) \to \phi(z)$ uniformly for $z \in X_0$
by the assumed continuity.

On the other hand, for $r \nearrow 1$, 
\[ f_r(z) = f(rz) \to f(z)
\] in
in $L^2(\T^2)$. Therefore, $f_r$ possesses two-sided limits in
$L^2(X_0)$ (i.e. for $r \searrow 1$ and $r \nearrow 1$) for any
compact subset $X_0 \subset X$.

The distributional edge-of-the-wedge theorem (Theorem B of Rudin
\cite{wR71}) now applies.  It requires that 
\[
\lim_{r \to 1} \int_{X} f_r(z) \psi(z) d\sigma(z)
\]
exist for every $\psi \in C^{\infty}_c(X)$.  The convergence of $f_r$
to $f$ in $L^2(X_0)$ on either side of any compact $X_0\subset X$ is
more than enough for this.  The conclusion of the edge-of-the-wedge
theorem is that $f$ has a holomorphic extension to a domain $\Omega_0$
containing $W^{+} \cup X \cup W^{-}$.  An important part of the
theorem is that the domain depends only on $W^{\pm}, X.$

Then, for each $x\in X,$ $f$ extends analytically to a neighborhood
$N_x$ of $x$, and so Proposition \ref{prop:rudin2} implies that $f$
extends analytically to an open set $U$ containing $X_1 \times \D$ and
$\D \times X_2$.  Again, from the proof of the theorem, it is clear
that $U$ depends only on the $\{N_x\}_{x \in X},$ which in turn
depended only on $W^{\pm}, X.$

As $f$ is already holomorphic
in $\D^2\cup (\E^2\setminus S)$ we may conclude that every $f \in
\Kphi$ is holomorphic in an open set 
\[
\Omega = \Omega_0 \cup U \cup \D^2 \cup (\E^2\setminus S)
\]
and meromorphic in $\Omega' = \Omega_0\cup U \cup \D^2 \cup \E^2$. 
\end{proof}

\begin{claim} Points of $\Omega$ are bounded point evaluations
  for $\Kphi, \Kphi^1, \Kphi^2$. \end{claim}
\begin{proof}
Again, we consider only $\Kphi$.  Let $B$ be the set of bounded point
evaluations of $\Kphi$ in $\Omega$.  It is clear that $\D^2 \subset B$
since points of $\D^2$ are bounded point evaluations of all of $H^2$. 
Also, $\E^2\setminus S \subset B$ by the definition of exterior values
of functions in $\Kphi$.  As a first step, we show that $B$ is
relatively closed in $\Omega$ and this will in particular show that $X
\subset B$.

Suppose $\{w^j\} \subset B$ and $w^j \to w \in \Omega$.  For each $f
 \in \Kphi$,
\[
\sup \{ |f(w^j)|: j \geq 0\} < \infty.
\]
 By the uniform boundedness principle, there is a constant $M$ such
 that
\[
|f(w^j)| =  | \langle f, G_{w^j} \rangle_{\Kphi} | \le M \|f
\|_{\Kphi}.
\]
for all $f \in \Kphi$ and $ j \ge 0.$ As each $f$ is holomorphic in
$\Omega$, $f(w^j) \rightarrow f(w)$ and so
\[
 |f(w) | \le M \|f \|_{\Kphi}.
\] 
Hence, $w \in B$ and  $B$ is a relatively closed subset of $\Omega$.

To show $B$ contains $\Omega_0$ we need to refer to the local
construction of $\Omega_0$ as in the continuous edge-of-the-wedge
theorem as proved in Rudin \cite{wR71}.  Refer to Proposition
\ref{prop:rudin} below.  Modulo rescaling and a change of coordinates,
the main point is that around any point $x \in X$, any $f \in \Kphi$
is extended to a neighborhood $N_x$ of $x$ in $\C^2$ via an integral
formula which only depends on the values of $f$ in a compact subset $K
\subset W^{+}\cup X \cup W^{-}$.  Now, every $f\in \Kphi$ is analytic
in a neighborhood of such a $K$ and so for all $f\in \Kphi$
\[
\sup \{|f(w)|: w \in K\} < \infty.
\]
By the uniform boundedness principle, there is a constant $M$ such
that for all $w\in K$ and $f \in \Kphi$
\[
|f(w)| \leq M \|f\|_{\Kphi}.
\]
Because of this, the values of any $f$ in $N_x$ are controlled by
$f$'s values in $K$ and hence by $M$ and $\|f\|_{\Kphi}$. Thus, 
the points of $\Omega_0$ (as constructed in the proof
of the edge-of-the-wedge theorem) are bounded point evaluations of
$\Kphi$. 

Now consider the points of $U,$ the set guaranteed by Proposition
\ref{prop:rudin2}. The set $U$ is constructed as a union of
neighborhoods of the points in $X_1 \times \D$ and $\D \times X_2$ as
follows:
\[
U =\bigcup_{z \in X_1 \times \D} N_z \cup \bigcup_{w \in \D \times X_2} N_w.
 \]
Specifically, fix $z=(x_1,z_2) \in X_1 \times \D$. Then, there is an
$x_2$ such that $(x_1,x_2) \in X$ and a neighborhood $N_x$ of $x$
(guaranteed by the edge-of-the-wedge theorem) such that each $f \in
\Kphi$ extends analytically to $N_x.$ Then, Proposition
\ref{prop:rudin2} guarantees a neighborhood $N_z$ of $z$ to which each
$f$ extends analytically.  It follows by the construction in the proof
that there is a compact set $K$ contained in $\D^2 \cup N_x$ such that
for all $z_0 \in N_z$ and $ f \in \Kphi$
\[
|f(z_0)| \le \sup_{w \in K} |f(w)|.
\]
We can again use the uniform boundedness principle to  
conclude that the points in $N_z$ are also bounded
point evaluations of $\Kphi.$  Note that
$\Omega$ is constructed in the proof essentially as
\[
\D^2\cup(\E^2\setminus S) \cup \bigcup_{x \in X} N_x  \cup
\bigcup_{z \in X_1 \times \D} N_z \cup \bigcup_{w \in \D \times X_2} N_w ,
\]
and so we have proven that the points of $\Omega$ are bounded
point evaluations of $\Kphi.$

 Finally, the reproducing kernel $G(z,w)$ can now be extended to be
 sesqui-analytic in $\Omega$. Similarly, the reproducing kernels of $\Kphi^1$ 
and $\Kphi^2$ can be extended to $\Omega \times \Omega,$ which implies
$F^1, F^2, E^1,E^2$ extend. Let $(A^1,A^2)$ be 
Agler kernels of $\phi$. By Corollary \ref{cor:Agker}, the points of 
$\Omega$ are bounded point evaluations of $\mcH(A^1)$ and 
$\mcH(A^2),$ and hence, $(A^1,A^2)$ extend to be sesqui-analytic in 
$\Omega$ as well.
\end{proof}

That concludes the proof of $(i) \Rightarrow (iii),$ and 
it is immediate that $(iii) \Rightarrow (ii)$. Now consider $(ii) \Rightarrow (i):$

\begin{proof}
Let $(A^1,A^2)$ be Agler kernels of $\phi$ such that the elements of 
$\mcH(A^1)$ and $\mcH(A^2)$ extend continuously to $X$. By definition,
\[
1-\phi(z)\phi(w)^* = (1-z_1\bar{w}_1)A^2(z,w) + (1-z_2\bar{w}_2)
A^1(z,w),
\]
for all $z,w \in \D^2.$ Since $\phi$ is an isometry almost everywhere on $\T^2$, we can
choose $w \in \D^2$ such that $\phi(w)$ is invertible.  (Recall that
we are assuming $\mathcal{V}$ is finite dimensional, so that $\phi$
converges to its boundary values radially almost everywhere.)  As
$A^1_{w}\nu \in \mcH(A^1)$ and $A^2_{w}\nu \in \mcH(A^2)$ 
both extend to be continuous on $X$ for all $\nu \in \mcV,$ so does
\[
\phi(z) = \left(-(1-z_1\bar{w}_1)A^2_w(z) - (1-z_2\bar{w}_2)
A^1_w(z)+1\right)(\phi(w)^*)^{-1}.
\]
\end{proof}

When showing $(i) \Rightarrow (iii)$, we made reference to the following 
construction of Rudin. Notice in particular that the integral formula for $F$ 
depends only on the values of $f$ on the ``wedge.''

\begin{prop} \label{prop:rudin}  \cite{wR71}*{pg 10} 
Let $E$ and $V$ be open cubes in $\mathbb{R}^2$ defined as follows:
\[
\begin{aligned}
E &=& \{ x : -6 <x_j<6 \text{ for } 1 \le j \le 2 \} \\ V &=& \{ y : 0
< y_j < 6 \ \text{ for } 1 \le j \le 2 \}.
\end{aligned}
\]
Define $R^+:=E +iV$ and $R^- := E-iV$. Assume $f$ is a function
continuous on $R^+ \cup E \cup R^-$ and holomorphic on $R^+ \cup
R^-$. Then, there exists a function $F$ holomorphic on $\D^2$ with $F
\equiv f$ on $\D^2 \cap (R^+ \cup E \cup R^-).$ Define 
\[ \psi(s, t)
:= \frac{s + \tfrac{ t}{c}}{1 +c s t}, \ \ \text{ where } c =
\sqrt{2}-1.\]
 Then $F$ is given by
\[
 F(\lambda) = \frac{1}{2 \pi} \int_{-\pi}^{\pi} f \big(
\psi(\lambda_1, e^{i \theta}), \psi(\lambda_2, e^{i\theta}) \big) \ d
\theta \ \ \text{ for } \lambda \in \D^2,
\]
 and for each pair $(\lambda, \theta)$, the point $\big(
 \psi(\lambda_1, e^{i \theta}), \psi(\lambda_2, e^{i\theta}) \big)$ is
 in $R^+ \cup E \cup R^-.$

\end{prop}

For convenience we recount the following definitions and proposition
from Rudin \cite{wR69}*{pg 97-99}, which were used above.

A boundary point $p$ of $\D^2$ is \emph{regular point} for a
holomorphic $f:\D^2\to \C$ if there is a neighborhood $N_p$ of $p$
where $f|_{N_p\cap \D^2}$ extends to be holomorphic on $N_p$.
Otherwise, $p$ is a \emph{singular point} of $f$.

\begin{prop} \label{prop:rudin2} \cite{wR69}*{Theorem 4.9.1 pg 98}
If $f$ is holomorphic in $\D^2$, $\beta \in \D$, and $(1,\beta)$ is a
singular point of $f$, then $(1,\eta)$ is a singular point of $f$ for
every $\eta \in \T$.
\end{prop}

The contrapositive implies that if $f$ is regular at $(1,1)$, then $f$
is regular at $(1,\beta)$ for each $\beta \in \D$.  It can be seen
from the proof in \cite{wR69} that if $f$ is holomorphic in a
neighborhood $N$ of $(1,1)$, then for each $\beta \in \D$ there is a
compact set $K \subset \D^2 \cup N$ (depending only on $N$, $\beta$
and not $f$) such that $|f(1,\beta)| \leq \max_K |f|$. In other words,
$(1,\beta)$ is in the holomorphically convex hull of $K$ in $\D^2 \cup
N$.

\section{Theorem \ref{thm:restrict} on restriction maps}

\begingroup
\def\thetheorem{\ref{thm:restrict}}
\begin{theorem} 
Let $\phi$ be a finite dimensional matrix valued inner function on
$\D^2$.  For almost every $t \in \T$, the map
\[
f \mapsto f(t,\cdot)
\]
embeds $\Kphi^1 \ominus Z_1 \Kphi$ and $\Kphi^1 \ominus \Kphi$
isometrically into $H^2(\T) \ominus \phi(t,\cdot) H^2(\T)$.
\end{theorem}
\addtocounter{theorem}{-1}
\endgroup

\begin{proof}
The proof is the same for $\Kphi^1 \ominus Z_1 \Kphi$ and $\Kphi^1
\ominus \Kphi$.  The key facts we use are that both spaces are
contained in $\Hphi^1$ and since 
\[
\Kphi^1\ominus \Kphi = (\Hphi^1\ominus \Kphi) \ominus
Z_1(\Hphi^1\ominus \Kphi)
\]
and
\[
\Kphi^1 \ominus Z_1\Kphi = \Hphi^1 \ominus Z_1 \Hphi^1 ,
\]
both of these spaces are orthogonal to their translates by $Z_1$.  

We provide the proof now for only $\Kphi^1 \ominus \Kphi$. By the
above observations, for any $f,g \in \Kphi^1\ominus \Kphi =
(\Hphi^1\ominus \Kphi) \ominus Z_1(\Hphi^1\ominus \Kphi)$, we have
that $f \perp Z_1^jg$ for all $j \in \mathbb{Z}$ except $j=0$.
Therefore, for $j\ne 0$
\[
 0 = \int_{\T} z_1^j \int_{\T}\ip{f(z)}{g(z)}_{\mcV} d\sigma(z_2)
 d\sigma(z_1),
\]
which implies that
\[
\int_{\T} \ip{f(z_1,z_2)}{g(z_1,z_2)}_{\mcV} d \sigma(z_2) = \ip{f}{g}
\]
for almost every $z_1 \in \T$.  

This shows the restriction map $f \mapsto f(t,\cdot)$ is an isometry
from $\Kphi^1 \ominus \Kphi$ to $H^2(\T)$ for almost every $t \in \T$.
(By separability, we can show that given a countable dense set
$\mathcal{D}$ in $\Kphi^1\ominus \Kphi$, for almost every $t$, every
$f\in \mathcal{D}$ possesses slices $f(t,\cdot)\in L^2(\T)$. Since we
will have an isometry on this dense set, it will extend to be
isometric on the whole space.)

Now, $\phi(t,\cdot)$ is inner for almost every $t \in \T$, and since
$f \in \Hphi^1$ implies $f(t,\cdot) \in H^2(\T)$ and $\phi(t,\cdot)^*
f(t,\cdot) \in L^2_{-}(\T)$ for almost every $t$, we see that for $f
\in \Hphi^1$
\[
f(t,\cdot) \in H^2(\T) \ominus \phi(t,\cdot) H^2(\T)
\]
for almost every $t \in \T$.  (Again, we could argue using separability
that we are only taking ``almost every $t$'' a countable number of
times.)

Therefore, for almost every $t \in \T$, $f \mapsto f(t,\cdot)$ is an
isometry from $\Kphi^1\ominus \Kphi$ into $H^2(\T) \ominus \phi(t,\cdot)
H^2(\T)$.  

\end{proof}

An obvious question is then:

\begin{question}\label{q:restrict} 
Is the restriction map above \emph{onto} for almost every $t$? 
\end{question}

We have been unable to resolve this but having some regularity on the
boundary allows us to prove a partial result.

\begin{prop} \label{prop:restrict}
If $\phi$ extends continuously to a rectangle $X = X_1\times X_2
\subset \T^2$,
then the restriction map
\[
f \mapsto f(t, \cdot)
\]
embeds $\Kphi^1\ominus \Kphi$ and $\Kphi^1\ominus Z_1 \Kphi$
isometrically \emph{onto} $H^2(\T) \ominus \phi(t,\cdot) H^2(\T)$ for
almost every $t \in X_1$.
\end{prop}

\begin{proof}
As before we treat the case $\Kphi^1 \ominus \Kphi$. By the regularity
results, $\phi$ and elements of $\Kphi^1$ extend analytically to a
domain $\Omega$ containing $\D^2, X , (X_1\times \D)$. In addition,
the reproducing kernels $F^1$ and $E^2$ are sesqui-analytic on $\Omega\times
\Omega$.  For $t \in X_1$, $\zeta, \eta \in \D$, we can use Theorem \ref{thm:fundagler}
to conclude
\[
\frac{1-\phi(t,\zeta)\phi(t,\eta)^*}{1-\zeta \bar{\eta}} =
F^1((t,\zeta),(t,\eta)).
\]
Therefore, for $\eta \in \D, \nu \in \mcV,$
\[
F^1_{(t,\eta)}\nu \mapsto \frac{1-\phi(t,\cdot)\phi(t,\eta)^*}{1-(\cdot)
  \bar{\eta}} \nu
\]
under the restriction map for $t\in X_1$.  Since $\phi(t,\cdot)$ is
inner for almost every $t$, this shows that the image of
$\Kphi^1\ominus \Kphi$ under this restriction map is a dense subset of
$H^2(\T)\ominus \phi(t,\cdot) H^2(\T)$ (namely the dense subset of the
span of reproducing kernels) for almost every $t\in X_1$.  Since the
restriction map is an isometry for almost every $t$, it must therefore
be a unitary for almost every $t\in X_1$.
\end{proof}

\begin{example} 
The example $\phi(z) = \frac{2z_1z_2-z_1-z_2}{2-z_1-z_2}$ shows that
we cannot have an isometry for \emph{every} $t$ in the restriction map
of Theorem \ref{thm:restrict}.  The simple reason is that $H^2(\T)
\ominus \phi(t,\cdot) H^2(\T)$ has dimension $1$ for all $t$ except
$t=1$ where it has dimension $0$, since $\phi(1,z_2) = -1$.
\EOEx
\end{example}

\section{A technical fact}

In the next sections we study rational inner functions.  The following
technical fact is essential.  A similar fact was needed in
\cite{gK10}.

\begin{prop} \label{prop:fp}
Suppose $p\in \C[z]$ has no zeros in $\D^2$, $f \in H^2$, and $f/p \in
L^2$. Then, $f/p \in H^2$.
\end{prop}
\begin{proof}
By Fubini's theorem, $f(\cdot,z_2) \in H^2(\T)$ for a.e. $z_2 \in \T$;
the same holds for $p(\cdot,z_2)$.  Recall that for a polynomial to be
\emph{outer}, in the sense of Hardy spaces in the disk, it is
necessary and sufficient that it have no zeros in $\D$. By Lemma
\ref{lem:outer} below, $p(\cdot,z_2)$ is outer for all but finitely
many $z_2 \in \T$ since $p$ has no zeros on $\D^2$, and therefore both
$f(\cdot,z_2)$ and $1/p(\cdot,z_2)$ are in the Smirnov class $N^{+}$
for a.e. $z_2 \in \T$.  As $N^{+}$ is an algebra,
$f(\cdot,z_2)/p(\cdot,z_2)$ is in $N^{+}$ for a.e. $z_2 \in \T$. Since
$N^{+} \cap L^2(\T) = H^2(\T)$ (see \cite{pD70}),
$f(\cdot,z_2)/p(\cdot,z_2) \in H^2(\T)$ for a.e. $z_2 \in \T$.  This
implies $f/p \perp L^2_{-\bullet}$.  A similar argument shows $f/p
\perp L^2_{\bullet -}$.  Therefore, $f/p \in H^2$.
\end{proof}

\begin{lemma} \label{lem:outer} 
If $p \in \C[z_1,z_2]=\C[z]$ has no zeros on $\D^2$, then for all $z_2
\in \T$ with at most a finite number of exceptions, $p(\cdot,z_2)$ has
no zeros on $\D$.
\end{lemma}
\begin{proof} For $0<r<1$ and $\zeta \in \T$, $p(\cdot,r\zeta)$ has no zeros in
  $\D$.  By Hurwitz's theorem, it follows that $p(\cdot,\zeta)$ is
  either identically zero or has no zeros in $\D$.  If
  $p(\cdot,\zeta)$ is identically zero, $p$ must have $z_2-\zeta$ as a
  factor.  As $p \in \C[z]$ can only have finitely many factors of
  this form, the claim follows.
\end{proof}

\section{Matrix rational inner functions and Theorem \ref{thm:rifdescrip}} \label{sec:mrif}
Suppose $\phi$ is a rational matrix inner function. Then we write $\phi$ as
\[
\phi(z) = \frac{Q(z)}{p(z)}, 
\]
where $p\in\C[z]$ is the least common multiple of the denominators of
the entries of $\phi$ after each entry is put into reduced form and has
no zeros in $\D^2$, and $Q \in \C^{N\times N}[z]$ is a matrix
polynomial satisfying
\[
Q^*Q = |p|^2 I = Q^t \bar{Q} \text{ on } \T^2.
\]

\begin{lemma} \label{finitezeros}
 With $\phi= Q/p$ as above, $p$ has finitely many zeros on $\T^2$.
\end{lemma}

\begin{proof}
For $Q/p$ to be holomorphic, it is necessary that $p$ have no
zeros in $\D^2$.  Every polynomial with no zeros in $\D^2$ may be
factored into two such polynomials $p=p_1p_2$ where $p_1$ has finitely
many zeros on $\T^2$, and each irreducible factor of $p_2$ has
infinitely many zeros on $\T^2$. We allow either factor to be a
constant. (This is the atoral-toral factorization of
Agler-McCarthy-Stankus \cite{AMS06}.)  Our claim is that $p_2$ is a constant. If
$p_2$ has some nontrivial irreducible factor $f$, then since $Q^*Q =
|p|^2I$, every entry of $Q$ vanishes on the zero set of $f$ and hence
every entry is divisible by $f$. (Two bivariate polynomials with
infinitely many common zeros must have factor in common.)  This
contradicts the fact that $p$ is the least common multiple of the
denominators of $\phi$.
\end{proof}

  Let $d =(d_1,d_2)$ be the maximal degree of $Q$. Then, $\tilde{Q}(z)
  = z^d Q(1/\bar{z})^*$ is a matrix polynomial and $\tilde{p}(z) = z^d
  \overline{p(1/\bar{z})}$ is a polynomial.  Set
\[
\tilde{\phi} = \frac{\tilde{Q}}{p},
\]
which is inner since $\tilde{Q}\tilde{Q}^* = Q^*Q = |p|^2I$ on $\T^2$.
Notice also that $\phi \tilde{\phi} =\frac{\tilde{p}}{p}I$.Define
\[
\begin{aligned}
\Pzp &= \{ q/p \in H^2: q \in \C^N[z], \deg q \leq (d_1-1,d_2-1) \} \\
\Pop &= \{ q/p \in H^2: q \in \C^N[z], \deg q \leq (d_1,d_2-1) \} \\
\Ptp &= \{ q/p \in H^2: q \in \C^N[z], \deg q \leq (d_1-1,d_2) \}. 
\end{aligned}
\]
Note that these spaces depend on the degree of $Q$ (and not
necessarily $p$).

We repeat Theorem \ref{thm:rifdescrip} here for convenience.

\begingroup
\def\thetheorem{\ref{thm:rifdescrip}}
\begin{theorem} If $\phi$ is a rational matrix inner function, then 
\[
\begin{aligned}
\Kphi &= \{f \in \Pzp: \tilde{\phi} f \in \Pzp \} \\
\Kphi^1 &= \{f \in \Pop: \tilde{\phi} f \in \Pop \} \\
\Kphi^2 &= \{f \in \Ptp: \tilde{\phi} f \in \Ptp \}.
\end{aligned}
\]
\end{theorem}
\addtocounter{theorem}{-1}
\endgroup

\begin{proof} 
We prove the theorem only for $\Kphi$; the claims for $\Kphi^1,
\Kphi^2$ are similar.  Set $d' = (d_1-1,d_2-1)$ and $Z^{d'} =
Z_1^{d_1-1} Z_2^{d_2-1}$.  We frequently use the fact that $f \in H^2$
and $Z^{d'} \bar{f} \in H^2$ implies $f$ is a polynomial of degree at
most $d'$.

If $g \in \Kphi$, then $g \in H^2$ and
\[
g = \frac{Q}{p} \overline{Z_1Z_2 h}
\]
for some $h \in H^2$.  Then,
\[
\begin{aligned}
pg &= Q \overline{Z_1Z_2 h} \in H^2 \\
Z^{d'} \overline{pg} &= Z^d \bar{Q} h = \tilde{Q}^t h \in H^2
\end{aligned}
\]
shows $q=pg$ is a polynomial of degree at most $d'$; i.e. $g \in
\Pzp$.  This also shows $\tilde{q}=\tilde{Q}^t h$ is a polynomial of
degree at most $d'$. Observe now that 
\[
Q^t\tilde{q} = Q^t \tilde{Q}^t h = \tilde{p} p h,
\]
which implies that
\[
\frac{Q^t\tilde{q}}{\tilde{p}} = ph \in H^2.
\]
On the other hand,
\[
Z^{d'} \overline{ph} = \frac{Q^* q}{\bar{Z}^d p} =
\frac{\tilde{Q} q}{p} = \tilde{\phi} q \in H^2,
\]
which shows $ph$ is a polynomial of degree at most $d'$; i.e. $h \in
\Pzp$.  In addition, $\tilde{\phi} q$ is a polynomial of degree at
most $d'$ and therefore,
\[
\tilde{\phi} \frac{q}{p} = \tilde{\phi} g \in \Pzp.
\]
  
Thus, we have proved
\[
\Kphi \subset \{f \in \Pzp: \tilde{\phi} f \in \Pzp\}
\]
and need to establish the opposite inclusion.  Suppose $f \in
\Pzp$ and
\begin{equation} \label{eq:opp}
\frac{\tilde{Q}}{p} f = \frac{r}{p}
\end{equation}
with $r \in \C[z]$ and $\deg r \leq d'$. Set $\tilde{r} = Z^{d'}\bar{r}.$
We must show $\phi^*f \in L^2_{--}$.  Observe that $Z^d Q^* f = r$
implies $Q^*f = \overline{Z_1Z_2\tilde{r}}$ and hence,
\[
\phi^* f = \bar{Z}_1\bar{Z}_2 \overline{\frac{\tilde{r}}{p}}. 
\]
By \eqref{eq:opp}, $r/p \in L^2$, which implies $\tilde{r}/p \in L^2$
and so by Proposition \ref{prop:fp}, $\tilde{r}/p$ is in
$H^2$. Therefore, $\phi^* f \in L^2_{--}$ as desired.
\end{proof}

By our definitions, in the scalar case, $\tilde{\phi} =\frac{p}{p} =
1$.

\begin{corollary} \label{cor:srifdescrip}
If $\phi = \tilde{p}/p$ is a scalar rational inner function, then
\[
\Kphi = \Pzp \qquad \Kphi^1 = \Pop \qquad \Kphi^2 = \Ptp.
\]
If $p$ has no zeros on $\T^2$, then 
\[
\Kphi = \{ q/p: \deg q \leq (d_1-1,d_2-1)\}
\]
and similarly for $\Kphi^1,\Kphi^2$.
\end{corollary}

As an interesting aside, note that for scalar $\phi$, $\Kphi = \{q/p:
\deg q \leq (d_1-1,d_2-1)\}$ if and only if $1/p \in L^2(\T^2)$ if and
only if $p$ has no zeros on $\T^2$.  The main thing to check is that
$1/p\in L^2(\T^2)$ implies $p$ has no zeros in $\T^2$.  Since $p$ has only
finitely many zeros on $\T^2$, this a local problem.  So, let us
assume $p(1,1)=0$ and prove
\[
\iint_{[-\epsilon,\epsilon]^2} \frac{1}{|p(e^{i\theta_1},
  e^{i\theta_2})|^2} d\theta_1 d\theta_2 = \infty.
\]
In this case, 
\[
p(z) = \sum_{1 \le | \alpha|\le n}
C_{\alpha}(1-z_1)^{\alpha_1}(1-z_2)^{\alpha_2},
\]
where $|\alpha| = \alpha_1 +\alpha_2,$ and then 
\[
|p(e^{i\theta_1},e^{i\theta_2})|^2 \leq \text{const}(|1-e^{i\theta_1}|^2 +
|1-e^{i\theta_2}|^2) \leq \text{const}(\theta_1^2+\theta_2^2).
\]
This shows
\[
\iint_{[-\epsilon,\epsilon]^2} \frac{1}{|p(e^{i\theta_1},
  e^{i\theta_2})|^2} d\theta_1 d\theta_2 \geq \text{const}
\iint_{[-\epsilon,\epsilon]^2} \frac{1}{\theta_1^2+\theta_2^2}
d\theta_1 d\theta_2,
\]
which diverges.

\section{Examples}

We use several examples to highlight differences between the
situation when $\phi$ is scalar rational inner and the situation when $\phi$ is 
matrix rational inner. 

First, if $\phi$ is scalar rational inner, continuous on
$\overline{\D^2}$, and has a unique Agler decomposition, then $\phi$
is a function of one variable. (See \cite{gK10}.) This result fails
when $\phi$ is matrix rational inner. Clearly, if $\phi(z)$ is diagonal with
functions of one variable alone on the diagonal, then $\phi$ has a
unique Agler decomposition.  This still holds if we replace $\phi$
with $U\phi U^*$ where $U$ is a constant unitary matrix. 

However, those are not the only matrix rational inner 
functions with unique decompositions. 

\begin{example} \label{ex:reguniquedecomp}
 Let
\[ 
\phi(z) = \frac{1}{2}\left ( 
\begin{array}{cc}
z_1(z_1 +z_2) & z_1(z_1-z_2) \\
z_1-z_2 & z_1+z_2 
\end{array}
\right).
\]
Then $\phi$ is matrix rational inner, and it is easy to show that
$\phi$ is not of the form $UD(z)U^*$, where $U$ is unitary
and $D(z)$ is diagonal with entries that are functions of either $z_1$
or $z_2$ alone. We use Theorem \ref{thm:rifdescrip} to calculate
$\Kphi.$ Let $f \in \Kphi.$ Then $\deg f \le (1,0),$ and we can
write \[ f(z) = \left(
\begin{array}{c}
a_1 +b_1z_1\\
a_2 +b_2z_1
\end{array}
\right),
\]
for constants $a_1,a_2,b_1,b_2.$ An easy calculation gives
\[
\tilde{\phi}(z) = 
\frac{1}{2}\left ( 
\begin{array}{cc}
z_1 +z_2 & z_1(z_2-z_1) \\
z_2-z_1 & z_1(z_1+z_2) 
\end{array}
\right).
\]
As $f \in \Kphi$, we have
\[
\tilde{\phi}(z) f =
\frac{1}{2}\left ( 
\begin{array}{cc}
z_1 +z_2 & z_1(z_2-z_1) \\
z_2-z_1 & z_1(z_1+z_2) 
\end{array}
\right)
\left(
\begin{array}{c}
a_1 +b_1z_1\\
a_2 +b_2z_1
\end{array}
\right)
= \left(\begin{array}{c}
c_1  +d_1z_1\\
c_2 +d_2z_2
\end{array}
\right),
\]
where $c_1, c_2, d_1, d_2$ are constants. Thus,
\[
\tfrac{1}{2}(z_1 +z_2)(a_1 +b_1z_1) + \tfrac{1}{2} z_1(z_2-z_1)( a_2 +b_2z_1) = c_1  +d_1z_1,
\]
and an examination of the coefficients implies $a_1=a_2=b_1=b_2 = 0.$ Thus,
$\Kphi = \{0\},$ and it follows from Corollary \ref{cor:unique} that $\phi$ has a
unique Agler decomposition. The decomposition is given by
 \[ \begin{aligned} 
 \left( \begin{array}{cc}
 1 & 0 \\
 0 & 1 
 \end{array}
 \right)
 &- \frac{1}{4}  \left( 
 \begin{array}{cc}
z_1( z_1 +z_2) & z_1(z_1-z_2) \\
 z_1-z_2 & z_1+z_2 
 \end{array}
 \right)
  \left( 
 \begin{array}{cc}
 \bar{w_1}(\bar{w}_1 +\bar{w}_2) & \bar{w}_1-\bar{w}_2 \\
 \bar{w}_1(\bar{w}_1-\bar{w}_2) & \bar{w}_1+\bar{w}_2 
 \end{array}
 \right) \\
 & = (1-z_1 \bar{w}_1) \left( \frac{1}{2}
 \left( \begin{array}{cc}
 0 & z_1 \\
 0 & 1
 \end{array}
 \right)
 \left( \begin{array}{cc}
 0 & 0 \\
 \bar{w}_1 & 1
 \end{array}
 \right)
 +
 \left( \begin{array}{cc}
 1 & 0 \\
 0 & 0
 \end{array}
 \right)^2 \right) \\
& \hspace{.2in}+
 (1-z_2 \bar{w}_2) \frac{1}{2}
  \left( \begin{array}{cc}
 0 & -z_1 \\
 0 & 1
 \end{array}
 \right)
 \left( \begin{array}{cc}
 0 & 0 \\
 -\bar{w}_1 & 1
 \end{array}
 \right).
 \end{aligned}
 \]
\EOEx
\end{example}

\begin{question} \label{q:unique}
Can one characterize the regular matrix rational inner
  functions with unique Agler decompositions?  What can be said in
  the non-regular case?
\end{question}

Now we consider another difference between the matrix and scalar
cases. For $\phi = Q/p$ rational inner, observe that $\deg Q$ is not
necessarily the exact degree of every entry of $Q.$ This discrepancy
breaks down some of the structure seen in the scalar rational inner
case. As shown in Corollary \ref{cor:srifdescrip}, for $\phi$ scalar
rational inner, $ \Pzp \subseteq \Hphi,$ and $\Hphi \cap \Pzp = \Kphi.$
However, the following example illustrates that neither relation holds
for an arbitrary rational matrix inner function.

\begin{example} Let $\phi_1(z) = \tfrac{3z_1z_2-z_1-z_2}{3-z_1-z_2}$
and $\phi_2(z) =z^2_1z^2_2$ and define 
\[\phi(z) = \left ( 
\begin{array}{cc}
\phi_1(z) & 0 \\
0 & \phi_2(z) 
\end{array}
\right). \]
Then $p(z) = 3-z_1-z_2$, and we can rewrite $\phi$ as 
\[
\frac{Q(z)}{p(z)} 
=
 \frac{1}{3-z_1-z_2}
\left(
\begin{array}{cc}
3z_1z_2-z_1-z_2 & 0 \\
0 & z^2_1z^2_2(3-z_1-z_2)
\end{array}
\right),
\]
so that $\deg Q = (3,3).$ Observe that 
\[
\Hphi = \Big \{ \Big (
\begin{array}{c}
f_1 \\
f_2
\end{array}
\Big ) 
: f_i \in \mcH_{\phi_i} 
\text{ for } i=1,2
\Big \},
\]
and we can calculate $\mcH_{\phi_i}$ from Proposition 4.8 in \cite{kB11} as follows:
\[
\begin{aligned}
\mcH_{\phi_1} &= \{ \tfrac{f}{p} \in H^2(\T^2) :  
\hat{f}(j_1,j_2) = 0 \text{ for } j_1 > 0 \text{ and } j_2 > 0 \} \\
\mcH_{\phi_2} &= \{ f \in H^2(\T^2) :  
\hat{f}(j_1,j_2) = 0 \text{ for } j_1 > 1 \text{ and } j_2 > 1 \}.
\end{aligned}
\]
It is almost immediate that $\Pzp \not \subseteq \Hphi.$  Specifically, one can show
\[
\Big ( \begin{array}{c}
q_1 \\
q_2 
\end{array} 
\Big) / p  \in \Pzp \cap \Hphi
\]
if and only if each term in $q_1$ has degree zero
in one variable and degree at most two in the other, and $q_2$ is of 
the form $p(z)r(z)$, where $\deg r \le (1,1).$  Thus, 
$\Pzp \cap \Hphi  \ne \Pzp. $

Using Theorem \ref{thm:rifdescrip}, one can show $\Kphi$ is the set 
\[ 
\Kphi = \Big \{ 
\Big ( \begin{array}{c}
q_1 \\
q_2  
\end{array} 
\Big ) / p :
 q_1 \in \C, \ q_2=p r, \text{ where } \deg r \le (1,1) \Big \}.
\]
Thus, $ \Pzp \cap \Hphi  \not \subseteq \Kphi.$

\EOEx
\end{example}

\begin{example} \label{nonextremeaglerdecomps}

There are still interesting questions in the scalar case.  First,
observe that the set of Agler kernels $(A^1,A^2)$ of a function $\phi$ 
is a convex set. Now, consider the space $C(\D^4) \times C(\D^4)$, the direct product of
the space of continuous functions on $\D^4$ with itself, endowed with the topology of
uniform convergence on compact sets. It is easy to show that the set of 
Agler kernels of $\phi$  is a compact subset
of $C(\D^4) \times C(\D^4)$, i.e. that every sequence has a subsequence that 
converges to a pair of Agler kernels of $\phi.$ 
Then we can apply the Krein-Milman theorem to conclude that 
the set of Agler kernels of $\phi$  is the closed convex hull of its extreme points.

Agler kernels $(A^1,A^2)$ are said to come from an
orthogonal decomposition if $A^1/(1-z_1 \bar{w}_1), A^2/(1-z_2 \bar{w}_2)$ 
are reproducing kernels of closed subspaces of $\Hphi$.  This is equivalent to $G^1,G^2$ in
Theorem \ref{thm:maxmin} being reproducing kernels of orthogonal
closed subspaces of $\Kphi$ whose direct sum is all of $\Kphi$.  Agler
kernels coming from an orthogonal decomposition are extreme points in
the set of Agler kernels of $\phi$. We shall
use the basic example $\phi(z) = z_1^2z_2$ to show that there can be
other extreme points in the set of Agler kernels, and we raise the
following question.

\begin{question} \label{q:extreme}
Given an inner $\phi$, can one describe the extreme
  points in the set of all Agler kernels associated to $\phi$?
\end{question}

For $\phi(z) = z_1^2z_2$, $\Kphi = \vee\{1,Z_1\},
\Kphi^1\ominus \Kphi = \vee\{Z_1^2\}, \Kphi^2\ominus \Kphi =
\vee\{Z_2,Z_1Z_2\}$, which imply that
\[
G = 1+z_1\bar{w}_1,\quad F^1 = z_1^2\bar{w}_1^2,\quad F^2 =
z_2\bar{w}_2(1+z_1 \bar{w}_1).
\] 
Theorem \ref{thm:maxmin} shows that the only way to construct Agler
kernels $(A^1,A^2)$ of $\phi$ is to choose positive kernels $G^1, G^2$
such that
\[
z_1^2\bar{w}_1^2 +(1-z_1\bar{w}_1)G^1 = A^1 \kgeq 0
\]
and
\[
z_2\bar{w}_2(1+z_1 \bar{w}_1) + (1-z_2\bar{w}_2)G^2 = A^2 \kgeq 0.
\]
Observe that $A^1,A^2$ only come from an orthogonal decomposition if
additionally, $G^1,G^2$ are kernels of subspaces of $\Kphi$. It is easy to
that the only possible $G^1,G^2$ coming from such subspaces are
\begin{enumerate}
\item $G^1 = G,  G^2 = 0.$
\item $G^1 = 0,  G^2 = G.$
\item $G^1 = (a+bz_1)\overline{(a+bw_1)}, G^2 =
  (\bar{b}-\bar{a}z_1)\overline{ (\bar{b}-\bar{a}w_1)},$ where \\
 $|a|^2+|b|^2=1$.  
\end{enumerate}

The first two possibilities can indeed occur.  The third possibility
only occurs when $a=0, |b|=1$.  To see this, note that in the third
possibility
\[
\frac{A^1}{1-z_1\bar{w}_1} = G^1 + \frac{z_1^2
  \bar{w}_1^2}{1-z_1\bar{w}_1}
\]
must be the reproducing kernel of a subspace, call it $S$, of $\Hphi$, which is
invariant under multiplication by $Z_1$.  So, if $G^1 =
(a+bz_1)\overline{(a+bw_1)}$, then $a+bZ_1 \in S$, and so
$aZ_1+bZ_1^2 \in S$. Since $Z_1^2 \in \Kphi^1\ominus \Kphi$, we must have
$aZ_1 \in S$.  Therefore, if $a \ne 0$, then $Z_1 \in S,$ which implies $1
\in S$.  This puts us back in case (1) above.  So, $a=0$ and $|b|=1$,
which really means the only possibility is $G^1 = z_1\bar{w}_1, G^2 =
1$.

Therefore the only Agler kernels coming from orthogonal decompositions
are 
\[
\begin{aligned}
(1) \ \ & A^1(z,w) = 1, && A^2(z,w) = z_2\bar{w}_2(1+z_1\bar{w}_1) \\
(2) \ \ & A^1(z,w) = z_1^2\bar{w}_1^2, && A^2(z,w) = 1+z_1\bar{w}_1 \\
(3) \ \ & A^1(z,w) = z_1\bar{w}_1, && A^2(z,w) = 1+ z_1z_2\bar{w}_1\bar{w}_2.
\end{aligned}
\]

Convex combinations of these kernels are of the form
\begin{equation} \label{eqconvexcomb}
\begin{aligned}
A^1(z,w) &= a + b z_1\bar{w}_1 + c z_1^2\bar{w}_1^2 \\
 A^2(z,w) &= az_2\bar{w}_2(1+z_1\bar{w}_1) + b (1+z_1z_2\bar{w}_1\bar{w}_2) + c
(1+z_1\bar{w}_1),
\end{aligned}
\end{equation}
where $a+b+c=1$, $a,b,c \geq 0$. On the other hand, the following is a pair of Agler kernels
$(A^1,A^2)$ which are not of this form (we evaluate on the diagonal to
save space).
\[
\begin{aligned}
A^1(z,z) &= \tfrac{1}{4}|1+z_1|^2 + \tfrac{1}{4}|z_1|^2|1-z_1|^2 \\
&= |z_1^2|^2 + (1-|z_1|^2)G^1(z,z),
\end{aligned}
\]
where $G^1(z,z) = \tfrac{1}{4}(|1+z_1|^2+2|z_1|^2)$, and
\[
\begin{aligned}
A^2(z,z) &= \tfrac{1}{2} + \tfrac{1}{2}|z_1z_2|^2 + \tfrac{1}{4}|1-z_1|^2
+ \tfrac{1}{4}|z_2|^2|1+z_1|^2 \\
&= |z_2|^2(1+|z_1|^2)+(1-|z_2|^2)G^2(z,z),
\end{aligned}
\]
where $G^2(z,z) = \tfrac{1}{4}(2+|1-z_1|^2)$.  The pair $(A^1,A^2)$ is
not of the form \eqref{eqconvexcomb} since $A^1,$ for instance, contains
a $z_1$ term, but there is no such term in \eqref{eqconvexcomb}.
\EOEx
\end{example}

\section{Review of dimensions in one variable}

Before we move on to determine the dimensions of certain canonical
subspaces in two variables, it helps to review what happens in the
matrix case in one variable.

If $\phi$ is a matrix rational inner function of one variable, the
space $\Hphi = H^2(\T) \ominus \phi H^2(\T)$ has dimension determined
by the degree of $\det \phi$.  Specifically, $\det \phi$ is a finite
Blaschke product, and the dimension of $\Hphi$ is the number of factors
in the Blaschke product:
\begin{equation} \label{eq:onevardim}
\dim H^2(\T) \ominus \phi H^2(\T) = \deg \det \phi.
\end{equation}
 This is known as the Smith-McMillan degree.

To prove this, we need the Smith Normal form (see \cite{HK71}*{Section
  7.4}).  Writing $\phi = Q/p$, there exist matrix polynomials $S,T$
with matrix polynomial inverses such that
\[
S^{-1} Q T^{-1} = D = \text{diag}(D_1, D_2, \dots, D_N)
\]
is a diagonal matrix polynomial where $D_i$ divides $D_{i+1}$.

Then, \[
\phi H^2 = \frac{1}{p} S D T H^2 = \frac{1}{p} S D H^2. 
\]

Now, $H^2 \ominus \phi H^2$ is isomorphic as a vector space to the
quotient $H^2/\phi H^2,$ which is in turn isomorphic to
\[
S^{-1}H^2/S^{-1}\phi H^2
= H^2/\frac{D}{p} H^2. 
\]
Since $D$ is diagonal, this space is an algebraic direct sum of scalar
spaces $H^2/\frac{D_j}{p} H^2$.  Each factor has dimension given by
the number of zeros of $D_j$ which lie in the unit disk.  Since $\det
S=s_0$, $\det T=t_0$ are nonzero constants (since these matrix
polynomials have matrix polynomial inverses), and since
\[
\det \phi = s_0t_0 \frac{\prod D_j}{p^N} 
\]
is a Blaschke product, the total dimension will be the number of zeros
of this Blaschke product, counting multiplicities.

\section{Theorem \ref{thm:rifdim} on dimensions of canonical subspaces}

We now assume $\phi = Q/p$ is an $N\times N$ matrix valued rational
inner function on $\D^2$ as in Section \ref{sec:mrif}.  The scalar
function $\det \phi = \frac{1}{p^N}\det Q$ is a rational inner
function on $\D^2$.  It therefore has a representation as
\[
\det \phi = \frac{\tilde{g}}{g},
\]
 where $g$ is a polynomial with no zeros on $\D^2$ and no factors in
 common with $\tilde{g},$ and $\tilde{g}(z) = z_1^{M_1} z_2^{M_2}
 \overline{g(1/\bar{z}_1, 1/\bar{z}_2)}$ for some integers $M_1,M_2$
 (see Rudin \cite{wR69} Section 5.2).  We necessarily have that $g$
 divides $p^N,$ which means $g$ has finitely many zeros on $\T^2$.

\begingroup
\def\thetheorem{\ref{thm:rifdim}}
\begin{theorem}
\[
\begin{aligned}
\dim \Kphi^1 \ominus Z_1 \Kphi &= \dim \Kphi^1 \ominus \Kphi = \deg_2 \tilde{g} \\
\dim \Kphi^2 \ominus Z_2 \Kphi &= \dim \Kphi^2 \ominus \Kphi = \deg_1 \tilde{g}. 
\end{aligned}
\]
\end{theorem}
\addtocounter{theorem}{-1}
\endgroup

The notation $\deg_j q$ refers to the degree of $q(z_1,z_2) \in
\C[z_1,z_2]$ in the variable $z_j$.

\begin{proof}
As the argument is the same for $\Kphi^j \ominus Z_j \Kphi$ and $\Kphi^j
\ominus \Kphi,$  we only address $\Kphi^1 \ominus \Kphi$.

There are only finitely many $t\in \T$ such that $p$ has a zero on the
line $\{t\}\times \overline{\D}$ by Lemmas \ref{lem:outer} and
\ref{finitezeros}.  If we choose $t$ such that $p$ has no zeros on
$\{t\}\times \overline{\D}$, then $\phi$ will be analytic in a
neighborhood of $\{t\}\times \T$.  Perturbing $t$ if necessary, by
Proposition \ref{prop:restrict}, the map
\[
\frac{f}{p} \mapsto \frac{f(t,\cdot)}{p(t,\cdot)}
\]
will map $\Kphi^1 \ominus \Kphi$ onto $H^2(\T)\ominus \phi(t,\cdot)
H^2(\T)$ isometrically. Hence,
\begin{equation} \label{restrictdim}
\dim \Kphi^1 \ominus \Kphi = \dim H^2(\T)\ominus \phi(t,\cdot)
H^2(\T).
\end{equation}

As $g$ has no zeros on the line $\{t\}\times \overline{\D}$ and no
zeros in $\D^2$, $\tilde{g}(t,\cdot)/g(t,\cdot)$ is a Blaschke product
of degree $\deg \tilde{g}(t,\cdot)$. (The degree could have been less
if $g(t,\cdot)$ had a zero on $\T$.) Further,
\begin{equation}\label{gdegree}
\deg \tilde{g}(t,\cdot) = \deg_2 \tilde{g}.
\end{equation}
To see this, write $M_2=\deg_2 \tilde{g}$ and
\[
g(z_1,z_2) = \sum_{j=0}^{M_2} g_j(z_1) z_2^j.
\]
We see that
\[
\tilde{g}(z_1,z_2) = \sum_{j=0}^{M_2} \tilde{g}_{M_2-j}(z_1) z_2^j,
\]
where we perform ``reflection'' of the one variable polynomials at the
appropriate degree $M_1$.  The top coefficient is $\tilde{g}_0(z_1)$,
which does not vanish for $z_1 = t,$ else $g(t,0) = g_0(t)$ would
vanish (which means $g$ would vanish on $\{t\}\times \overline{\D},$
where it does not).  Hence, $\tilde{g}(t,z_2)$ has degree precisely
$M_2$ in $z_2$.

Combining \eqref{gdegree}, \eqref{eq:onevardim}, and
\eqref{restrictdim}, we have
\[
\deg_2 \tilde{g} = \deg \tilde{g}(t,\cdot) = \dim H^2(\T)\ominus
\phi(t,\cdot) H^2(\T) = \dim \Kphi^1 \ominus \Kphi,
\]
as desired.

\end{proof}

These dimension results also hold for reproducing kernel Hilbert spaces
associated to more general Agler decompositions.

\begin{remark} Assume $(A^1,A^2)$ are Agler kernels of $\phi$ such 
that the reproducing kernel Hilbert spaces with kernels 
\[
\frac{A^1(z,w)}{1-z_1\bar{w}_1} \text{ and } \frac{A^2(z,w)}{1-z_2\bar{w}_2} 
\]
are closed subspaces of $H^2.$ Then, $\mcH(A^1)$ and $\mcH(A^2)$
are orthogonal to their translates by $Z_1$ and $Z_2$ respectively.
Moreover, $\mcH(A^1) \subseteq \Hphi^1$ and $\mcH(A^2) \subseteq
\Hphi^2.$ The subspaces discussed in Theorem \ref{thm:rifdim} are clearly
special cases of these general reproducing kernel Hilbert spaces, and
the arguments in Theorem \ref{thm:restrict}, Proposition
\ref{prop:restrict}, and Theorem \ref{thm:rifdim} are valid for these
more general $\mcH(A^1)$ and $\mcH(A^2)$. Specifically,
\[
\begin{aligned}
\dim \mcH(A^1)&= \deg_2 \tilde{g} \\
\dim \mcH(A^2) &= \deg_1 \tilde{g}. 
\end{aligned}
\]

On the other hand, for \emph{general} Agler kernels the best that can
be said is 
\[
\begin{aligned}
\dim \Kphi^1 &\geq \dim \mcH(A^1) &\geq \deg_2 \tilde{g} \\
\dim \Kphi^2 &\geq \dim \mcH(A^2) &\geq \deg_1 \tilde{g}.
\end{aligned}
\]
The upper bounds follow from Corollary \ref{cor:Agker} while the lower
bounds are the content of Corollary \ref{cor:minag}, which we now
prove.
\end{remark}

\begin{proof}[Proof of Corollary \ref{cor:minag}]
By Theorem \ref{thm:maxmin} there exists a positive kernel $G^1 \kleq
G$ so that
\[
A^1(z,w) = F^1(z,w) + (1-z_1\bar{w}_1)G^1(z,w).
\]
By Proposition \ref{prop:restrict} we may choose $t \in \T$ such that
the restriction map $f \mapsto f(t,\cdot)$ maps $\Kphi^1\ominus \Kphi$
one-to-one and onto $H^2(\T)\ominus \phi(t,\cdot)H^2(\T)$.  This shows
the kernels $F^1_{(t,\eta)}v$, where $\eta$ varies over $\D$ and $v
\in \mcV$, are dense in $\Kphi^1\ominus \Kphi$ since any $f$
orthogonal to all such kernels would vanish on the set
$\{(t,\eta):\eta \in \D\}$.  Such an $f$ would then map to zero under
the restriction map, contradicting the fact that it is one-to-one and
onto.

Finally,
\[
A^1((t,z_2),(t,w_2)) = F^1((t,z_2),(t,w_2))
\]
and this shows $\dim \mcH(A^1)$ is at least the dimension of the space
spanned by $F^1_{(t,\eta)}(t,\cdot)v$ where $\eta \in \D$, $v\in
\mcV$. This is the same as the dimension of the space spanned by
$F^1_{(t,\eta)}v$ where $\eta \in \D$, $v\in \mcV$, since the
restriction map is bijective.  Therefore,
\[
\dim \mcH(A^1) \geq \dim \Kphi^1\ominus \Kphi = \deg_2\tilde{g}.
\]
The proof for $A^2$ is similar.
\end{proof}

\begin{proof}[Proof of Corollary \ref{cor:transfer}]
Given our Agler decomposition
\[
1-\phi(z)\phi(w)^* = (1-z_1\bar{w}_1)E^2(z,w) +
(1-z_2\bar{w}_2)F^1(z,w),
\]
we rearrange to yield
\[
1+z_1\bar{w}_1E^2(z,w) + z_2\bar{w}_2 F^1(z,w) = \phi(z)\phi(w)^* +
E^2(z,w) + F^1(z,w)
\]
and write 
\[E^2(z,w) = \sum_{j=1}^{d_1} E_j(z)E_j(w)^* \ \ \text{ and }   \ \ F^1(z,w) =
\sum_{j=1}^{d_2} F_j(z)F_j(w)^*, 
\]
 where $\{E_1,\dots, E_{d_1}\}$ is an
orthonormal basis for $\Kphi^2\ominus Z_2\Kphi$ and $\{F_1,\dots,
F_{d_2}\}$ is an orthonormal basis for $\Kphi^1\ominus \Kphi$.

It simplifies notation to write $E(z) = (E_1(z),\dots,E_{d_1}(z))$, $F(z) = (F_1(z),
\dots, F_{d_2}(z))$. Then we have
\[
E^2(z,w) = E(z)E(w)^*, F^1(z,w) = F(z)F(w)^*.
\]
It can be shown that the map defined for each row vector $v \in
\C^{N}$ and $z \in \D^2$ by:
\[
(v,z_1 v E(z), z_2 v F(z))^t \mapsto (v\phi(z), vE(z), vF(z))^t
\]
extends to a unitary $U$ from $\C^N\oplus \C^{d_1}\oplus \C^{d_2}$ to
itself. This is called a ``lurking isometry argument.''  Since this is
a standard trick, we refer the reader to the proof of Lemma 6.7 in
\cite{gK10} where more details on this trick are provided.

We then obtain the formula
\begin{equation} \label{lurkingisometry}
U\begin{bmatrix} I \\ z_1E(z)^t \\  z_2F(z)^t\end{bmatrix}
= \begin{matrix} \ & \begin{matrix} \C^N & \C^{|d|} \end{matrix}
  \\
\begin{matrix} \C^N \\ \C^{|d|} \end{matrix} & \begin{pmatrix} A &
  B \\ C & D \end{pmatrix} \end{matrix} \begin{bmatrix} I \\ z_1E(z)^t
  \\  z_2F(z)^t\end{bmatrix}
= \begin{bmatrix} \phi(z)^t \\ E(z)^t \\ F(z)^t \end{bmatrix}.
\end{equation}
It is now possible to solve for $\phi(z)^t$ and see that it has a
representation as in Corollary \ref{cor:transfer}.  Indeed,
\[
\begin{aligned}
A + Bd(z) (E(z), F(z))^t &= \phi(z)^t \\
C + Dd(z) (E(z), F(z))^t &= (E(z),F(z))^t
\end{aligned}
\]
and this implies 
\begin{equation} \label{formulaEF}
(E(z),F(z))^t = (I-Dd(z))^{-1}C
\end{equation}
which implies 
\[
\phi(z)^t = A+Bd(z)(I-D d(z))^{-1}C. 
\]

Of course, we could have applied the above argument to $\phi(z)^t$ to
see that $\phi(z)$ has such a representation as well.  (This annoyance
stems from the fact that we prefer to have column-vector-valued spaces
of functions.)

If we had a representation of $\phi$ using a \emph{smaller} unitary
$U$, say of size $(N+k_1+k_2)\times (N+k_1+k_2)$, then it is possible
to reverse the arguments to get Agler kernels from equation
\eqref{formulaEF} whose dimensions are $(k_1,k_2)$, which is not
possible.
\end{proof}

\section{Theorem \ref{thm:threevar}, an application to three variables}

Theorem \ref{thm:threevar} can be slightly rephrased as follows:

\begingroup
\def\thetheorem{\ref{thm:threevar}}
\begin{theorem} If $p \in \C[z_1,z_2,z_3]$ has degree $(n,1,1)$ and no
  zeros on $\overline{\D^3}$, then 
\[
|p(z)|^2 - |\tilde{p}(z)|^2 = \sum_{j=1}^{3} (1-|z_j|^2) SOS_j(z,z),
\]
where $SOS_2$ and $SOS_3$ are sums of two squares, while $SOS_1$ is a
sum of $2n$ squares.
\end{theorem}
\addtocounter{theorem}{-1}
\endgroup

Recounting all of the details of \cite{gKsacrifott} would take us too
far afield, so we shall only sketch the proof.  It is mostly a matter
of inserting Corollary \ref{cor:transfer} into the proper place in the
proof.

\begin{proof}[Sketch of proof:]
Write $p(z_1,z_2,z_3) = a(z_1,z_2)+b(z_1,z_2)z_3,$ where $a, b \in
\C[z_1,z_2]$ have degree at most $(n,1)$.  Define 
\[
\tilde{a}(z_1,z_2) = z_1^n z_2\overline{a(1/\bar{z}_1, 1/\bar{z}_2)},
\tilde{b}(z_1,z_2) = z_1^n z_2\overline{b(1/\bar{z}_1, 1/\bar{z}_2)}.
\]
Using the stability of $p$, it is possible to show $a$ and
$a+\tilde{b}$ have no zeros in $\overline{\D^2}$.  It is shown in
\cite{gKsacrifott} that for $z_1,z_2\in \T$ we may factor
\[
|a(z_1,z_2)|^2 - |b(z_1,z_2)|^2
\]
as a sum of two squares as follows:
\[
|a(z_1,z_2)|^2 - |b(z_1,z_2)|^2 = \|E(z_1,z_2)\|^2 = |E_1(z_1,z_2)|^2 +
|E_2(z_1,z_2)|^2,
\]
where $E = (E_1,E_2)^t \in \C^2[z_1,z_2]$ is a (column) vector valued
polynomial of degree at most $(n,1)$. (The $E$ here has no relation to
the reproducing kernel $E$ from earlier parts of the current paper.
We are attempting to match the notation of \cite{gKsacrifott}.)  Let
\[
\tilde{E}(z_1,z_2) = z_1^n z_2 \overline{E(1/\bar{z}_1, 1/\bar{z}_2)}.
\]

One can show
\[
V := \frac{1}{a} \begin{bmatrix} \tilde{b} & \tilde{E}^t \\ E &
  \frac{E\tilde{E}^t - a(\tilde{a} + b)I}{a + \tilde{b}} \end{bmatrix}
\]
is a $3\times 3$ matrix valued inner function with the property that
\begin{equation} \label{property}
V(z_1,z_2) \begin{bmatrix} p(z_1,z_2,z_3) \\ z_3
  E(z_1,z_2) \end{bmatrix} = \begin{bmatrix} \tilde{p}(z_1,z_2,z_3)
  \\ E(z_1,z_2) \end{bmatrix}
\end{equation}
for $(z_1,z_2,z_3) \in \D^3$. In order to use Theorem \ref{thm:rifdim}, 
we need to compute $\det V$. It is a direct calculation that
\begin{equation} \label{determinant}
\det V = \frac{\tilde{a}(\tilde{a} + b)}{a(a+\tilde{b})},
\end{equation}
which has degree at most $(2n,2)$ since $a,b$ have degree at most
$(n,1)$. A quick way to see this is to observe that  
\[
V\begin{bmatrix} a & b & 0 \\ 0 & E_1 & -\tilde{E}_2 \\ 0 & E_2 &
\tilde{E}_1 \end{bmatrix} = 
\begin{bmatrix} \tilde{b} & \tilde{a} & 0 \\ 
E_1 & 0 & \frac{\tilde{a}+b}{a+\tilde{b}} \tilde{E}_2 \\
E_2 & 0 & -\frac{\tilde{a}+b}{a+\tilde{b}} \tilde{E}_1 \end{bmatrix}.
\]
The first two columns on the right side are the result of
\eqref{property}, while the third column on the right comes from
$(-\tilde{E}_2, \tilde{E}_1)\tilde{E} = 0$.  If we now take the
determinant of both sides, we get
\[
(\det V) a (E_1\tilde{E}_1+E_2 \tilde{E}_2) =
\tilde{a}\frac{\tilde{a}+b}{a+\tilde{b}} (E_1\tilde{E}_1+E_2
\tilde{E}_2),
\]
which implies \eqref{determinant}.

By Corollary \ref{cor:transfer} or equation \eqref{lurkingisometry}
applied to $\phi=V^t$, there is a $(3+2n+2)\times (3+2n+2)$ unitary
$U$ such that
\[
U\begin{bmatrix} I \\ z_1 G_1(z) \\ z_2 G_2(z) \end{bmatrix}
= \begin{bmatrix} V(z) \\ G_1(z) \\ G_2(z) \end{bmatrix},
\]
where $G_1$ is a $2n\times 3$ matrix valued rational function, and
$G_2$ is a $2\times 3$ matrix valued rational function.
 
If we multiply this equation on both sides by $Y = \begin{bmatrix} p
  \\ Z_3 E \end{bmatrix},$ we get via \eqref{property} that
\[
U \begin{bmatrix} p \\ Z_3 E \\ Z_1H_1 \\ Z_2 H_2 \end{bmatrix}
= \begin{bmatrix} \tilde{p} \\ E \\ H_1 \\ H_2 \end{bmatrix},
\]
where $H_1 = G_1 Y$ and $H_2 = G_2Y$. Since $U$ is unitary, if we take
norms (pointwise) of both sides and rearrange, we are left with
\[
|p(z)|^2 - |\tilde{p}(z)|^2 = \sum_{j=1,2} (1-|z_j|^2) \|H_j(z)\|^2 +
(1-|z_3|^2) \|E(z)\|^2.
\]
Now, $H_1$ is $2n\times 1$ and $H_2$ is $2\times 1$, so that
$\|H_1\|^2$ is a sum of $2n$ squares, $\|H_2\|^2$ is a sum of $2$ and
$\|E\|^2$ is a sum of 2.  It is shown in \cite{gKrifitsac} that the
entries of $H_1$ and $H_2$ must be polynomials (specifically, see
Claim 2 on page 351 of \cite{gKrifitsac}).
\end{proof}

\section{Final comments}
Let us end by highlighting further areas of work and a few questions
in addition to the previously-asked Questions \ref{q:optimal},
\ref{q:restrict}, \ref{q:unique}, and \ref{q:extreme}.

It would be interesting to completely describe the canonical spaces
for more exotic inner functions or even the inner functions of Ahern
\cite{pA72}, which are essentially rational in one of the variables.

For inner functions of more than two variables, certain decompositions
of $\Hphi$ analogous to Theorem \ref{thm:fundagler} were discovered in
\cite{GKVW09} (and this was followed up in \cite{gK11}), but it
remains a challenge to form a useful decomposition of $\Hphi$
involving ``small'' subspaces like $\Kphi, \Kphi^1,
\Kphi^2$.  It would especially be interesting to decompose the
reproducing kernel for $\Hphi$ for rational inner $\phi$ using only
finite dimensional spaces and their shifts.  As already mentioned, an
Agler decomposition as in Theorem \ref{thm:threevar} will not hold more
generally, even for rational inner functions. Even when such a
decomposition does hold, it not clear whether or not 
decompositions can be constructed naturally from orthogonal sums of
associated Hilbert spaces.  The construction in the proof of Theorem
\ref{thm:threevar} is most likely not of this form.

\end{document}